\documentclass[10pt,oneside,english]{amsart}
\usepackage[T1]{fontenc}
\usepackage[latin9]{inputenc}
\usepackage{geometry}
\usepackage{amsthm}
\usepackage{amstext}
\usepackage{amssymb}

\makeatletter
\numberwithin{figure}{section}
\theoremstyle{plain}
\newtheorem{thm}{\protect\theoremname}
  \theoremstyle{remark}
  \newtheorem{rem}[thm]{\protect\remarkname}
  \theoremstyle{definition}
  \newtheorem{example}[thm]{\protect\examplename}
  \theoremstyle{plain}
  \newtheorem{question}[thm]{\protect\questionname}
  \theoremstyle{plain}
  \newtheorem{prop}[thm]{\protect\propositionname}
  \theoremstyle{definition}
  \newtheorem{defn}[thm]{\protect\definitionname}

\usepackage{graphicx}
\usepackage{color}

\makeatother

\usepackage{babel}
  \providecommand{\definitionname}{Definition}
  \providecommand{\examplename}{Example}
  \providecommand{\propositionname}{Proposition}
  \providecommand{\questionname}{Question}
  \providecommand{\remarkname}{Remark}
\providecommand{\theoremname}{Theorem}

\begin{document}

\author{Adrien Dubouloz}

\address{IMB UMR5584, CNRS, Univ. Bourgogne Franche-Comté, F-21000 Dijon,
France.}

\email{adrien.dubouloz@u-bourgogne.fr}

\author{Frédéric Mangolte}

\address{LUNAM Université, LAREMA, Université d\textquoteright{}Angers. }

\email{frederic.mangolte@univ-angers.fr}

\thanks{This project was partially funded by ANR Grant \textquotedbl{}BirPol\textquotedbl{}
ANR-11-JS01-004-01. }

\title{Real frontiers of fake planes}
\begin{abstract}
In \cite{DM15}, we define and partially classify fake real planes, that is, minimal complex surfaces with conjugation whose real locus is diffeomorphic to the euclidean real plane $\mathbb{R}^{2}$. Classification results are given up to biregular isomorphisms and up to birational diffeomorphisms. In this note, we describe in an elementary way numerous examples of fake real planes and we exhibit examples of such  planes of every Kodaira dimension $\kappa\in \{-\infty,0,1,2\}$ which are birationally diffeomorphic to $\mathbb{R}^{2}$. 
\end{abstract}
\maketitle

\section*{Introduction}

A \emph{complexification} of a real smooth $\mathcal{C}^{\infty}$-manifold $M$ is a real algebraic manifold $S$ (see terminological conventions at the beginning of Section~\ref{sec.gene}) whose real locus is diffeomorphic to $M$: $S(\mathbb{R})\approx M$. Some manifolds such as real projective spaces $\mathbb{RP}^{n}$ and real euclidean affine spaces $\mathbb{R}^{n}$ have natural algebraic complexifications, given by the complex projective and affine spaces $\mathbb{CP}^{n}$ and $\mathbb{C}^{n}$ respectively. But these also admit infinitely many other complexifications, and it is a natural problem to try to classify them up to appropriate notions of equivalence. 
We focus on the case when $M=\mathbb{R}^{2}$ which is already surprisingly involved. Consider the following two examples:

\begin{enumerate} 
\item Start from $S_0=\mathbb{C}^{2}_{u,v}$ with the usual conjugation $(u,v)\mapsto (\overline{u},\overline{v})$. We have $S_0(\mathbb{R})=\mathbb{R}^{2}$. Blowing-up a pair of conjugated non real points of $\mathbb{C}^{2}$, we get a surface $S_1$ which is not isomorphic to $\mathbb{C}^{2}$, actually  not even affine as it contains proper curves, but $S_1(\mathbb{R})\approx\mathbb{R}^{2}$. 
\item Start from $\mathbb{CP}^{2}_{x:y:z}$ and choose a pair of conjugated lines $L,\overline{L}$ meeting at a point on the \emph{line at infinity} $L_{\infty}:=\{z=0\}$. Let $S_2$ be the complement in $\mathbb{CP}^{n}$ of the union $L \cup\overline{L} \cup L_{\infty}$. Then $S_2$ is isomorphic to $\mathbb{C}\setminus\{\text{a pair of conjugated points}\}\times\mathbb{C}$ and $S_2(\mathbb{R})\approx\mathbb{R}^{2}$. 
\end{enumerate}

In the two cases above, the complexification is \emph{topologically far} from $\mathbb{C}^{2}$. So we will seek for \emph{topologically minimal} complexifications of $\mathbb{R}^{2}$ which we call \emph{fake real planes}:

\begin{defn} 
\label{defn.fake}%
A nonsingular quasi-projective complex algebraic surface $S$ endowed with an anti-holomorphic involution is called a \emph{fake real plane} if: 
\begin{enumerate} \item $S$ is a \emph{real plane}: $S(\mathbb{R})\approx\mathbb{R}^{2}$; \item $S$ is \emph{topologically minimal}: $H_k(S;\mathbb{Q})=0$ for all $k\geq 1$; \item $S$ is \emph{fake}: $S$ is not biregularly isomorphic 
to $\mathbb{C}^{2}$ as a real algebraic surface. 
\end{enumerate} 
\end{defn}

Every fake real plane $S$ is affine and rational over $\mathbb{R}$, see \S\ \ref{subsec.generalities} below. 
In the projective setting, we obtain complexifications of $\mathbb{RP}^{2}$ in the same way than described in the first example above by blowing-up pairs of conjugated non real points of $\mathbb{CP}^{2}$. Among the rational projective complexifications of $\mathbb{RP}^{2}$, the only minimal one is $\mathbb{CP}^{2}$ because it is the only one with Picard number $\rho=1$. In fact, $\mathbb{CP}^{2}$ is the only topologically minimal complexification of $\mathbb{RP}^{2}$: there exists other smooth complex surfaces with the same homology as $\mathbb{CP}^{2}$, but none of them admits a real structure \cite{KK02}.
In the light of these observations, a first natural question is:

\begin{question}
\label{question.fake}%
Is there any fake real plane at all? \end{question}

An affirmative answer to Question~\ref{question.fake} is given in the article \cite{DM15}, where we introduce and partially classify fake real planes.
In the present paper, as an application of general classification results in \cite{DM15}, we describe in an explicit way numerous examples. 

Let $S$ and $S'$ be two complexifications of a given topological surface $M$. We say that $S$ and $S'$ are \emph{$\mathbb{R}$-biregularly birationally equivalent} if their real loci $S(\mathbb{R})$ and $S'(\mathbb{R})$ have isomorphic Zariski open neighborhoods in $S$ and $S'$, respectively. Equivalently the surfaces $S(\mathbb{R})$ and $S'(\mathbb{R})$ are \emph{birationally diffeomorphic}, that is: 
\begin{defn} 
\label{defn.birdif}%
\label{defn.rectif}%
Let $S$ and $S'$ be real algebraic surfaces. Their real loci $S(\mathbb{R})$ and $S'(\mathbb{R})$ are \emph{birationally diffeomorphic} if there is a diffeomorphism $f \colon S(\mathbb{R}) \to S'(\mathbb{R})$ which extends as a real birational map $\psi \colon S \dashrightarrow S'$ whose indeterminacy locus does not intersect $S(\mathbb{R})$, and such that the indeterminacy locus of $\psi^{-1}$ does not intersect $S'(\mathbb{R})$. 

A real algebraic surface $S$ is \emph{rectifiable} if its real locus $S(\mathbb{R})$ is birationally diffeomorphic to $\mathbb{R}^{2}$.
\end{defn}

It is well-known that every rational projective complexification of $\mathbb{RP}^{2}$ is $\mathbb{R}$-biregularly birationally equivalent to  $\mathbb{RP}^{2}$, even dropping the topological minimality condition, see \cite{Ma15}. Thus next natural questions are:
\begin{question}  
\label{question.rectif}%
 Is there any rectifiable fake real plane? 
\end{question}

\begin{question} 
\label{question.rectif.all}%
 Let $S$ be a fake real plane. Is $S$ rectifiable? 
 \end{question}
 
 Again the answer to Question~\ref{question.rectif} is affirmative; examples of rectifiable fake real planes of Kodaira dimensions $-\infty$ and $0$ were given in \cite{DM15}. In the present paper, we exhibit infinite families of rectifiable fake real planes in every Kodaira dimension.
In spite of this, Question~\ref{question.rectif.all} remains open.

\bigskip
The idea of this paper emerges thanks to the conference \emph{Frontiers of rationnality} which took place in Spitsbergen (Norway) in July 2014. The word "frontier" in our title is an attempt to give credit to this exceptional event.

\section{Generalities on fake real planes}
\label{sec.gene}%
\subsubsection*{Terminological conventions} 
\begin{enumerate} \item A \emph{real algebraic manifold} of dimension $n$ is a quasi-projective complex algebraic manifold $S$ of complex dimension $n$ endowed with an anti-holomophic involution whose set of fixed points is called the \emph{real locus} and denoted by $S(\mathbb{R})$. A \emph{real map} is a complex map commuting with involutions. A \emph{real algebraic surface} is a real algebraic manifold of dimension $2$. The manifolds $S$ and $S(\mathbb{R})$ are assumed to be endowed with their euclidean topology.
\item A \emph{topological surface} is a real $2$-dimensional $\mathcal{C}^\infty$-manifold. By our convention, a real algebraic surface $S$ is nonsingular; as a consequence, if nonempty, the real locus $S(\mathbb{R})$ gets a natural structure of a topological surface when endowed with the euclidean topology. Furthermore $S(\mathbb{R})$ is compact if $S$ is projective. \end{enumerate}

\subsection{Generalities on topologically minimal real planes }
\label{subsec.generalities}%
By virtue of results of Fujita \cite{Fu82} and Gurjar-Sashtri \cite{GuP97,GuPS97}
every smooth $\mathbb{Q}$-acyclic complex surface $S$ is affine
and rational. If $S$ is real, then its admits a real completion $S\hookrightarrow(V,B)$
into a smooth real projective surface $V$ with connected
real boundary curve $B=V\setminus S$. In the case where $S$ is a
real plane, the connectedness of $S(\mathbb{R})\approx\mathbb{R}^{2}$
implies that $V$ is a rational complex surface with connected real
locus, hence an $\mathbb{R}$-rational real projective surface, i.e.
a real surface admitting a real birational map $V\dashrightarrow\mathbb{CP}^{2}$
to $\mathbb{CP}^{2}$ equipped with its standard real structure. The
free abelian group $\mathbb{Z}\langle B\rangle$ generated by the
irreducible components of $B$ and the divisor class group $\mathrm{Cl}(V)$
of $V$ both inherits a structure of $G$-module for the group $G=\{1,\sigma\}\simeq\mathbb{Z}_{2}$
generated by the real structure $\sigma$ on $V$. The inclusion $j\colon B\hookrightarrow V$
gives rise to a homomorphism $j_{*}\colon \mathbb{Z}\langle B\rangle\rightarrow\mathrm{Cl}(V)$
of $G$-module, hence to a induced homomorphism $H^{2}(j_{*})\colon H^{2}(G,\mathbb{Z}\langle B\rangle)\rightarrow H^{2}(G,\mathrm{Cl}(V))$
of $\mathbb{Z}_{2}$-vector spaces between the Galois cohomology groups
$H^{2}(G,M)=\mathrm{Ker}(\mathrm{id}_{M}-\sigma)/\mathrm{Im}(\mathrm{id}_{M}+\sigma)$,
$M=\mathbb{Z}\langle B\rangle,\mathrm{Cl}(V)$. The next theorem will
be the most useful for the constructions presented below: 
\begin{thm}
\label{thm:Top-carac}
Let $(V,B)$ be a pair consisting of an $\mathbb{R}$-rational
projective real surface $V$ and a real curve $B\subset V$. Then
the surface $S=V\setminus B$ is an $A$-acyclic, where $A=\mathbb{Z}$
or $\mathbb{Q}$, real plane if and only if the conditions are satisfied:

a) $V$ and $B$ are connected and simply connected and $j_{*}\otimes_{\mathbb{Z}}A\colon \mathbb{Z}\langle B\rangle\otimes_{\mathbb{Z}}A\rightarrow\mathrm{Cl}(V)\otimes_{\mathbb{Z}}A$
is an isomorphism. 

b) The real locus of $B$ is nonempty and $H^{2}(j_{*})\colon H^{2}(G,\mathbb{Z}\langle B\rangle)\rightarrow H^{2}(G,\mathrm{Cl}(V))$
is an isomorphism. \end{thm}
\begin{proof}
The first assertion is essentially a rephrasing of a classical criterion
established first by Ramanujam \cite{Ram71} and Fujita \cite{Fu82}
which asserts that the complex surface $S$ is $A$-acyclic if and only if $V$ and $B$
are connected and the homomorphism $H_{2}(B;A)\rightarrow H_{2}(V;A)$
induced by the inclusion $B\hookrightarrow V$ is an isomorphism.
Indeed, $H_{2}(B;A)$ is a free abelian group isomorphic to $\mathbb{Z}\langle B\rangle\otimes_{\mathbb{Z}}A$
via the map which associates to every irreducible component of $B$ its
fundamental class in homology while the isomorphism $H_{2}(V;A)\simeq\mathrm{Cl}(V)\otimes_{\mathbb{Z}}A$
follows from that fact that since $V$ is rational, the cycle map
$\mathrm{Cl}(V)\rightarrow H_{2}(V,\mathbb{Z})$ which associates
to every irreducible complex curve $D\subset V$ its fundamental class
is an isomorphism. The second assertion essentially follows from the
combination of the well-known fact that a relatively compact topological
surface $M\subset\overline{M}$ with connected boundary $\overline{M}\setminus M$
is diffeomorphic to $\mathbb{R}^{2}$ if and only if it is connected
and $\mathbb{Z}_{2}$-acyclic with the cycle map construction due
to Borel-Haefliger \cite{BH61}, see \cite[Section 2]{DM15}. \end{proof}
\begin{rem}
\label{Rk:Top-carac-rem}

1) The curve $B$ in the previous proposition
need not be an SNC divisor on $V$, but the criterion implies in particular
that if $S=V\setminus B$ is $A$-cyclic then its inverse image $\tau^{-1}(B)$
in a log-resolution $\tau\colon V'\rightarrow V$ of the pair $(V,B)$ defined
over $\mathbb{R}$ is a real tree of rational curves, that is a tree
of rational curves with an induced action of the real structure $\sigma$
on $V'$. It follows in particular that the real locus of $B$ is
either empty or a connected union of curves homeomorphic to a circle. 

2) In the case where $j_{*}\otimes_{\mathbb{Z}}\mathbb{Q}\colon \mathbb{Z}\langle B\rangle\otimes_{\mathbb{Z}}\mathbb{Q}\rightarrow\mathrm{Cl}(V)\otimes_{\mathbb{Z}}\mathbb{Q}$
is an isomorphism, one infers from the long exact sequence of relative
homology for the pair $(V,B)$ that $H_{1}(S,\mathbb{Z})$ is a torsion
group isomorphic to $\mathrm{Cl}(V)/\mathrm{Im}(j_{*})$.
\end{rem}

\begin{example}
\label{Ex:conic-counterexample} The complement $S$ of a smooth real
conic $B$ in $\mathbb{CP}^{2}$ is a $\mathbb{Q}$-acyclic real surface
which is a not a real plane. It can be seen directly that $S(\mathbb{R})$
is either diffeomorphic to $\mathbb{R}\mathbb{P}^{2}$ if $B(\mathbb{R})=\emptyset$
or to the disjoint union of $\mathbb{R}^{2}$ with a Möebius band
otherwise. In the setting of Theorem~\ref{thm:Top-carac} above, the image of the generator
$[B]$ of $\mathbb{Z}\langle B\rangle\simeq\mathbb{Z}\cdot[B]\simeq H_{2}(B;\mathbb{Z})$
by $j_{*}\colon \mathbb{Z}\langle B\rangle\rightarrow\mathrm{Cl}(\mathbb{C}\mathbb{P}^{2})\simeq H_{2}(\mathbb{CP}^{2};\mathbb{Z})\simeq\mathbb{Z}\cdot[\ell]$
where $[\ell]$ denotes the class of a real line, is equal to $2\cdot[\ell]$.
So $S$ is $\mathbb{Q}$-acyclic, with $H_{1}(S;\mathbb{Z})\simeq\mathbb{Z}_{2}$,
but not $\mathbb{Z}$-acyclic. Furthermore, since in this particular
situation we have isomorphisms $H^{2}(G,\mathbb{Z}\langle B\rangle)\simeq\mathbb{Z}\langle B\rangle\otimes_{\mathbb{Z}}\mathbb{Z}_{2}$
and $H^{2}(G,\mathrm{Cl}(V))\simeq\mathrm{Cl}(V)\otimes_{\mathbb{Z}}\mathbb{Z}_{2}$,
$H^{2}(j_{*})\colon H^{2}(G,\mathbb{Z}\langle B\rangle)\rightarrow H^{2}(G,\mathrm{Cl}(V))$
is then the trivial map.
\end{example}

\section{Fake planes of negative Kodaira dimension }

It turns out that $\mathbb{C}^{2}$ equipped with its standard real
structure is the only $\mathbb{Z}$-acyclic real plane $S$ of negative
Kodaira dimension. Indeed, a complex $\mathbb{Z}$-acyclic surface
of negative Kodaira dimension is isomorphic to $\mathbb{C}^{2}$ by
virtue of \cite{MS80} while every real structure on $\mathbb{C}^{2}$
is isomorphic to the standard one, as a consequence of \cite{Kam75}.
In this section, we briefly review general geometric properties of
$\mathbb{Q}$-acyclic fake planes. Then as a particular instance of
a large class of such planes which are known since \cite{DM15} to be $\mathbb{R}$-biregularly
birationally equivalent to $\mathbb{C}^{2}$, see Theorem~\ref{thm:Real-plane-NegKappa-Rectif} below, we construct a pair of 
non isomorphic $\mathbb{Q}$-acyclic fake planes of negative Kodaira
dimension with the same homology groups, whose real loci are both
birationally diffeomorphic to $\mathbb{R}^{2}$.

\subsection{Basic properties of $\mathbb{Q}$-acyclic fake planes of negative
Kodaira dimension}

By virtue of \cite{MS80}, the negativity of the Kodaira dimension
of a smooth complex affine surface $S$ is equivalent to the existence
of an $\mathbb{A}^{1}$-fibration $\rho\colon S\rightarrow C$, that is,
a fibration with general fibers isomorphic to $\mathbb{C}$, over
a smooth curve $C$. One direction is clear for if $\rho\colon S\rightarrow C$
is such a fibration, then $C$ contains a nonempty Zariski open subset
$C_{0}$ over which $\rho$ restricts to a trivial bundle $C_{0}\times\mathbb{C}$,
and so $\kappa(S)\leq\kappa(C_{0}\times\mathbb{C})=-\infty$ by Iitaka's
easy addition formula. When $S$ is in addition real, there is no
reason in general that the existing $\mathbb{A}^{1}$-fibration $\rho\colon S\rightarrow C$
is a real map: for instance, the complement of a smooth conic $B$ in
$\mathbb{CP}^{2}$ with empty real locus is a real surface of negative
Kodaira dimension without any real $\mathbb{A}^{1}$-fibration $\rho\colon S\rightarrow C$.
Indeed, if such a fibration existed then the closure in $\mathbb{CP}^{2}$
of its fiber over a general real point of $C$ would be a real rational
curve intersecting $B$ in a unique point, necessarily real, which
is impossible. The following theorem shows in particular that this
phenomenon does not occur for real planes: 
\begin{thm}
\label{thm:RealPlan-negKappa-Carac} $($\cite[Theorem 4.1]{DM15}$)$
For a smooth affine real surface $S$ the following are equivalent:

1) $S$ is a $\mathbb{Q}$-acyclic real plane of negative Kodaira dimension.

2) $S$ admits a real $\mathbb{A}^{1}$-fibration $\rho\colon S\rightarrow\mathbb{C}$
whose closed fibers are all isomorphic to $\mathbb{C}$ when equipped
with their reduced structure and whose fibers over the real locus of $\mathbb{C}$
have odd multiplicities.\end{thm}
\begin{example}
\label{ex:Fib-hypersurfaces-negKappa} 
Let $s\geq1$, let $[m]=(m_{1},\ldots,m_{s})$
be a collection of integers $m_{i}\geq2$, let $[p]=(p_{1},\ldots,p_{s})$
be a collection of odd integers $p_{i}\geq3$ and let $r_{1},\ldots,r_{s}\in\mathbb{R}$
be a collection of pairwise distinct real numbers. Then the real surface
$S_{s,[m][p]}\subset\mathbb{C}^{3}$ defined by the equation 
\[
\prod_{i=1}^{s}(x-r_{i})^{m_{i}}z=\sum_{i=1}^{s}\prod_{j\neq i}(x-r_{j})y^{p_{i}}+\prod_{i=1}^{s}(x-r_{i})
\]
is a smooth $\mathbb{Q}$-acyclic fake plane of negative Kodaira dimension.
Indeed, the smoothness of $S_{s,[m][p]}$ follows for instance from
the Jacobian criterion. The restriction to $S_{s,[m][p]}$ of the
projection $\mathrm{pr}_{x}$ is a real $\mathbb{A}^{1}$-fibration
$\pi\colon S_{s,[m][p]}\rightarrow\mathbb{C}$ with irreducible fibers,
restricting to a trivial bundle over $\mathbb{C}\setminus\{r_{1},\ldots,r_{s}\}$
and whose fiber over each of the points $r_{i}$, $i=1,\ldots,s$,
has odd multiplicity $p_{i}$. So $S_{s,[m][p]}$ is a $\mathbb{Q}$-acyclic
real plane of negative Kodaira dimension by the previous theorem. 
\end{example}
The following partial result concerning the classification of $\mathbb{Q}$-acyclic
real plane of negative Kodaira dimension up to $\mathbb{R}$-biregular
equivalence was obtained in \cite{DM15}:

\begin{thm}
\label{thm:Real-plane-NegKappa-Rectif} Every $\mathbb{Q}$-acyclic
real plane $S$ of negative Kodaira dimension admitting a real $\mathbb{A}^{1}$-fibration
$\rho\colon S\rightarrow\mathbb{C}$ with at most one degenerate fiber is
$\mathbb{R}$-biregularly birationally equivalent to $\mathbb{C}^{2}$. 
\end{thm}
As a consequence, the real locus of every surface $S_{1,m_{1},p_{1}}$
as in Example \ref{ex:Fib-hypersurfaces-negKappa} is birationally
diffeomorphic to $\mathbb{R}^{2}$. We don't know whether $\mathbb{Q}$-acyclic
real plane of negative Kodaira dimension $S$ admitting real $\mathbb{A}^{1}$-fibrations
$\rho\colon S\rightarrow\mathbb{C}$ with more than one degenerate fiber
are $\mathbb{R}$-biregularly birationally equivalent to $\mathbb{C}^{2}$,
in particular:
\begin{question}
Is the real locus of a surface $S_{s,[m][p]}$ with $s\geq2$ as in
Example \ref{ex:Fib-hypersurfaces-negKappa} birationally diffeomorphic
to $\mathbb{R}^{2}$ ? 
\end{question}

\subsection{Two examples of $\mathbb{Q}$-acyclic fake planes with real loci
birationally diffeomorphic to $\mathbb{R}^{2}$ }

\subsubsection{Construction and first properties}

\indent\newline\indent 1) First we let $S$ be the complement of
a real cuspidal cubic $B$ in $V=\mathbb{CP}^{2}$. Since $B$ is
equivalent in $\mathrm{Cl}(V)$ to three times the class of a real
line $\ell\subset\mathbb{CP}^{2}$, the $\mathbb{Q}$-acyclicity of
$S$ immediately follows from Theorem \ref{thm:Top-carac} a) applied
to the pair $(V,B)$, and we have $H_{1}(S;\mathbb{Z})\simeq\mathbb{Z}_{3}$
by Remark \ref{Rk:Top-carac-rem}.2. The fact that $S$ is a real
plane follows from b) in the same theorem after noting that similarly
as in Example \ref{Ex:conic-counterexample}, $H^{2}(G,\mathbb{Z}\langle B\rangle)\simeq\mathbb{Z}\langle B\rangle\otimes_{\mathbb{Z}}\mathbb{Z}_{2}\simeq\mathbb{Z}_{2}\cdot[B]$
and $H^{2}(G,\mathrm{Cl}(\mathbb{CP}^{2}))\simeq\mathrm{Cl}(\mathbb{CP}^{2})\otimes_{\mathbb{Z}}\mathbb{Z}_{2}\simeq\mathbb{Z}_{2}\cdot[\ell]$.
Alternatively, one can observe that $B(\mathbb{R})$ is homeomorphic
to a simple closed curve in $\mathbb{RP}^{2}$ whose homotopy class
is a generator of $\pi_{1}(\mathbb{RP}^{2})$, so $\mathbb{RP}^{2}\setminus B(\mathbb{R})$
is homeomorphic to an open disc, hence is diffeomorphic to $\mathbb{R}^{2}$.
The restriction to $S$ of the rational pencil $V\dashrightarrow\mathbb{CP}^{1}$
generated by $B$ and three times its tangent $T$ at its unique singular
point restricts to a real $\mathbb{A}^{1}$-fibration $\rho\colon S\rightarrow\mathbb{C}$
with a unique degenerate fiber of multiplicity $3$ consisting of
the intersection of $T$ with $S$. The real minimal resolution of
the pencil $V\dashrightarrow\mathbb{CP}^{1}$ dominates the real minimal
resolution $\tau\colon W\rightarrow V$ of the pair $(V,B)$, and is obtained
from it by blowing-up three times the intersection point of the proper
transform of $B$ with the successive total transforms of the exceptional
locus of $\tau$. We denote by $\beta\colon \tilde{V}\rightarrow V$ the
so constructed surface and we denote by $C$ the last exceptional
divisor produced by this sequence of blow-ups. The dual graph of the
total transform of $B\cup T$ in $\tilde{V}$ is depicted in Figure~\ref{fig:cusp-cubic-init} below. 

\begin{figure}[!htb]
\input{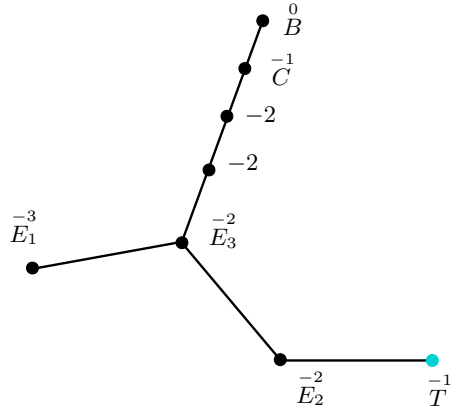} 
\caption{Dual graph of the
proper transform of $B\cup T$ in $\tilde{V}$.}  
\label{fig:cusp-cubic-init}
\end{figure}

2) Next we let $S'$ be the smooth real affine cubic surface in $\mathbb{C}^{3}$
defined by the equation $x^{2}z=y^{3}-x$. Taking projective closure in $\mathbb{C}\mathbb{P}^{3}$ 
with homogeneous coordinates $[x:y:z:t]$, we view $S'$ as the complement in the normal real cubic surface $V'=\{x^{2}z-y^{3}+xt^{2}=0\}$ 
of the real rational cuspidal hyperplane section $B'=V'\cap\{t=0\}$. Note that $V'$ has a unique singular
point of type $D_{4}$ at $[0:0:1:0]$. Since the map $\mathbb{R}^{2}\rightarrow S'$,
$(x,z)\mapsto(x,\sqrt[3]{x^{2}z+x},z)$ is an homeomorphism between
$\mathbb{R}^{2}$ and the real locus of $S'$, it follows that $S'$
is a real plane. The divisor class group $\mathrm{Cl}(V')$ of $V'$
is isomorphic to $\mathbb{Z}$ generated by the closure $T'$ in $V'$
of the real line $\{x=y=0\}\subset S'$. The divisor class group $\mathrm{Cl}(W')$
of the real minimal log-resolution $\tau'\colon W'\rightarrow V'$ of the
pair $(V',B')$ is isomorphic to $\mathbb{Z}^{6}$ generated by the five exceptional
divisors of $\tau'$ and the proper transform of $T'$, and since
$B'\sim3T'$ in $\mathrm{Cl}(V')$, it follows that for an appropriate
choice of bases, the map $j_{*}\colon \mathbb{Z}\langle\tau^{-1}(B')\rangle\rightarrow\mathrm{Cl}(W')$
is represented by a matrix $M$ of the form 
\[
M=\left(\begin{array}{cc}
\mathrm{id}_{5} & *\\
0 & 3
\end{array}\right).
\]
So by virtue of Theorem \ref{thm:Top-carac} a) applied to the real
projective completion $(W',\tau^{-1}(B'))$ of $S'$, $S'$ is $\mathbb{Q}$-acyclic,
with $H_{1}(S';\mathbb{Z})\simeq\mathbb{Z}_{3}$ by Remark \ref{Rk:Top-carac-rem}.2. The restriction to $S'$ of the projection $\mathrm{pr}_{x}$
is a real $\mathbb{A}^{1}$-fibration $\rho'\colon S'\rightarrow\mathbb{C}$
having the line $\{x=y=0\}$ as a unique degenerate fiber of multiplicity
$3$. So $\kappa(S')=-\infty$. The real minimal resolution of the
pencil $V'\dashrightarrow\mathbb{CP}^{1}$ induced by $\rho'\colon S'\rightarrow\mathbb{C}$
dominates the real minimal resolution $\tau'\colon W'\rightarrow V'$ of
the pair $(V',B')$, and is obtained from it by blowing-up the intersection
point of the proper transform of $B'$ with the exceptional locus
of $\tau'$. We denote by $\beta'\colon \tilde{V}'\rightarrow V$ the so
constructed surface and we denote by $C'$ the last exceptional divisor
produced by this sequence of blow-ups. The dual graph of the total
transform of $B'\cup T'$ in $\tilde{V}'$ is depicted in Figure~\ref{fig:E6-cubic-init} below. 

\begin{figure}[!htb]
\input{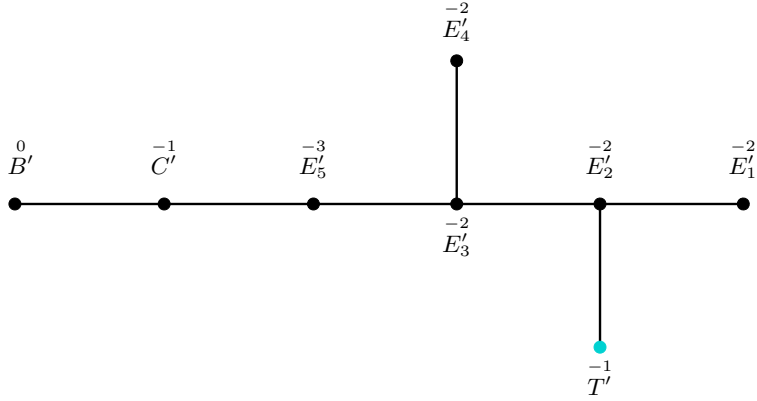} 
\caption{Dual graph of the
proper transform of $B'\cup T'$ in $\tilde{V}$.}  
\label{fig:E6-cubic-init}
\end{figure}

\begin{prop}
The surfaces $S$ and $S'$ are non isomorphic fake planes of negative
Kodaira dimension.
\end{prop}

\begin{proof}
That $S$ and $S'$ are fake planes of negative Kodaira dimension
follows from the construction above. The fact they are non isomorphic
as complex surfaces can be deduced as follows. First since the dual
graphs of the minimal resolutions of the pairs $(V,B)$ and $(V',B')$
are not chains, it follows from \cite{Be83} that the $\mathbb{A}^{1}$-fibrations
$\rho\colon S\rightarrow\mathbb{C}$ and $\rho'\colon S'\rightarrow\mathbb{C}$
constructed above are unique, up to composition by automorphisms of
the base $\mathbb{C}$. This implies in particular that any
isomorphism $\Phi$ between $S$ and $S'$ is an isomorphism of fibered
surfaces, that is, there exists an isomorphism $\phi\colon \mathbb{C}\stackrel{\sim}{\rightarrow}\mathbb{C}$
such that $\rho'\circ\Phi=\phi\circ\rho$. Every such isomorphism
then admits a unique extension to a birational map $\overline{\Phi}\colon \tilde{V}\dashrightarrow\tilde{V}'$
with a unique proper base point, supported at the intersection of
$C$ with the proper transform of $B$, restricting to an isomorphism
between the complements $\tilde{V}\setminus\beta_{*}^{-1}(B)$ and
$\tilde{V}'\setminus(\beta')_{*}^{-1}(B')$ of the proper transforms
of $B$ and $B'$ respectively and mapping $C$ isomorphically onto
$C'$ (see e.g. \cite{BD11} where these birational maps are called
fibered modifications). But the fact that the irreducible components
of $\beta^{-1}(B)\setminus\beta_{*}^{-1}(B)$ and $\beta^{-1}(B')\setminus\beta_{*}^{-1}(B')$
intersecting $C$ and $C'$ respectively have different self-intersections
prevents the existence of any such birational map. So $S$ and $S'$
are not isomorphic. \end{proof}
\begin{rem}
The surfaces $S$ and $S'$ do not only have the same homology group
but also the same fundamental group $\pi_{1}\simeq\mathbb{Z}_{3}$.
In fact, it can be shown that the complex algebraic threefolds $S\times\mathbb{A}^{1}$
and $S'\times\mathbb{A}^{1}$ are isomorphic, so that $S$ and $S'$
are in particular homotopically equivalent. A direct computation
reveals further that the fundamental groups at infinity of $S$ and
$S'$ (see e.g. \cite[§ 4.9 p. 246]{MiyBook} for the definition and
an algorithm for the computation of these groups) are both isomorphic
to $\mathbb{Z}_{9}$. We do not know whether $S$ and $S'$ are homeomorphic
as real $4$-manifolds or not. 
\end{rem}

\subsubsection{Birational diffeomorphisms }\label{sec:DiffBir}

Here we show that the real loci of $S$ and $S'$ are not only diffeomorphic
to $\mathbb{R}^{2}$ but actually birationally diffeomorphic to it. 

1) In the minimal real log-resolution $\tau\colon W\rightarrow V$
of the pair $(V,B)$, the proper transform of any pair $(\ell,\overline{\ell})$ of general
non-real complex conjugate lines in $V=\mathbb{CP}^{2}$ passing through
the singular point $p_{0}$ of $B$ consists of a pair of complex
conjugate rational $0$-curves intersecting the proper transform of
$B$ transversally in a pair of non-real complex conjugate points
$(q,\overline{q})$. Choosing two distinct such general pairs $(\ell_{i},\overline{\ell}_{i})$,
$i=1,2$, we let $\theta_{1}\colon W\dashrightarrow W_{1}$ be the real
birational map consisting of the blow-up of the corresponding two
pairs of points $(q_{i},\overline{q}_{i})$, $i=1,2$, followed by
the contraction of the proper transforms of $\ell_{i}$ and $\overline{\ell}_{i}$,
$i=1,2$ (see Figure~\ref{fig:cusp-cubic-rectif} below). By construction, $\theta$ restricts to a diffeomorphism
$W(\mathbb{R})\approx W_{1}(\mathbb{R})$, the dual graph of the
proper transform of $\tau^{-1}(B)$ by $\theta_{1}$ is depicted in 
Figure~\ref{fig:cusp-cubic} below.

\begin{figure}[!htb]
\input{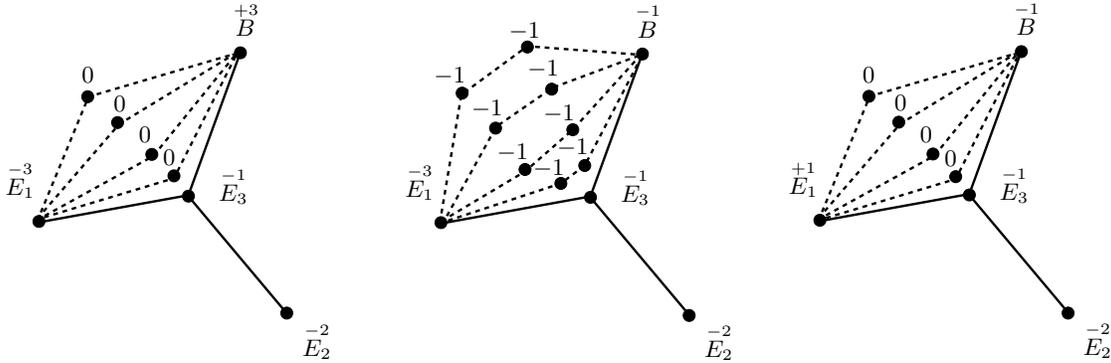} 
\caption{Elementary transformations.}  
\label{fig:cusp-cubic-rectif}
\end{figure}

\begin{figure}[!htb]
\input{cusp-cubic-f.tex} 
\caption{Dual graph of the
proper transform of $\tau^{-1}(B)$ by $\theta_{1}$.}  
\label{fig:cusp-cubic}
\end{figure}

Now let $\theta_{2}\colon W_{1}\dashrightarrow W_{2}$ be the real birational
map obtained by first blowing-up the point $E_{1}\cap E_{3}$, with
exceptional divisor $D_{1}$ and then contracting successively the proper
transforms of $B$, $E_{3}$ and $E_{2}$. 
By counting the number of points blown-up and the number of curves contracted during the whole process,
we infer that the Picard rank of $W_2$ is equal to $2$. So $W_2$ is a Hirzebruch surface in which 
the proper transforms of $D_{1}$ and $E_{1}$ are respectively 
a fiber and a section with self-intersection $1$ of a $\mathbb{P}^{1}$-bundle structure on $W_2$.
It follows that $W_2\simeq \mathbb{F}_{1}$ and that the proper transform of $D_1$ is a fiber of the unique 
$\mathbb{P}^{1}$-bundle structure $\pi_{1}\colon \mathbb{F}_{1}\rightarrow\mathbb{CP}^{1}$.
The restriction of $\theta_{2}\circ\theta_{1}$ to $S$ is a real
birational map $f\colon S\dashrightarrow\mathbb{F}_{1}\setminus(E_{1}\cup D_{1})\simeq\mathbb{C}^{2}$
inducing a diffeomorphism $S(\mathbb{R})\approx\mathbb{R}^{2}$.

2) For $S'$, the construction is very similar although of a slightly
different flavour. We first observe that the proper transform in the
real minimal resolution $\tau'\colon W'\rightarrow V'$ of the pair $(V',B')$
of a pair of general non-real complex conjugate hyperplane sections
$(\ell',\overline{\ell}')$ of $V'=\{x^{2}z-y^{3}+xt^{2}=0\}$ of
the form $V'\cap\{y-at=0\}$ and $V'\cap\{y-\overline{a}t=0\}$, $a\in \mathbb{C}\setminus \mathbb{R}$, consists
of a pair of complex conjugate rational $0$-curves intersecting the
proper transform of $B$ transversally in a pair of non-real complex
conjugate points $(q',\overline{q}')$. For every such pair $(\ell',\overline{\ell}')$,
the real birational map $\theta_{1}'\colon W'\dashrightarrow W_{1}'$ consisting
of the blow-up of $q'$ and $\overline{q}'$ followed followed by the
contraction of the proper transforms of $\ell'$ and $\overline{\ell}'$
restricts to a diffeomorphism $W'(\mathbb{R})\approx W_{1}'(\mathbb{R})$.
The dual graph of the proper transform of $(\tau')^{-1}(B')$ by $\theta_{1}'$
is depicted in Figure~\ref{fig:E6-cubic}.

\begin{figure}[!htb]
\input{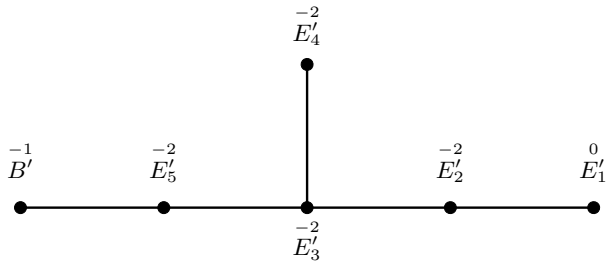} 
\caption{Dual graph of the
proper transform of $(\tau')^{-1}(B')$ by $\theta_{1}'$.}  
\label{fig:E6-cubic}
\end{figure}

Letting $\theta_{2}'\colon W_{1}'\rightarrow W_{2}'$ be the real birational
morphism consisting of the successive contractions of $B'$, $E_{5}'$,
$E_{3}'$ and $E_{4}'$, the resulting surface $W_{2}'$ is isomorphic
to $\mathbb{CP}^{1}\times\mathbb{CP}^{1}$ in which the proper transforms
of $E_{1}'$ and $E_{2}'$ are fibers say of the first and second
projection respectively. The restriction of $\theta_{2}'\circ\theta_{1}'$
to $S'$ is a real birational map $f'\colon S'\dashrightarrow\mathbb{CP}^{1}\times\mathbb{CP}^{1}\setminus(E_{1}'\cup E_{2}')\simeq\mathbb{C}^{2}$
inducing a diffeomorphism $S'(\mathbb{R})\approx\mathbb{R}^{2}$. 

\section{Fake planes of Kodaira dimension $0$ }

By virtue of \cite[Theorem 4.7.1 (1), p. 244]{MiyBook}, there is
no smooth complex $\mathbb{Z}$-acyclic surface of Kodaira dimension
$0$. Smooth complex $\mathbb{Q}$-acyclic surfaces of Kodaira dimension
$0$ do exist, and are completely classified after the work of Fujita
and Kojima (see also \cite[Chapter 3, §4]{MiyBook}). In this section,
we present certain families of fake planes of Kodaira dimension $0$,
some of which being $\mathbb{R}$-biregularly birationally equivalent
to $\mathbb{C}^{2}$.

\subsection{Real models of Fujita's $H[-k,k]$ surfaces}

In the Hirzebruch surface $\pi_{2p}\colon \mathbb{F}_{2p}\rightarrow\mathbb{CP}^{1}$,
$p\geq1$, with negative section $C_{0}\simeq\mathbb{CP}^{1}$ of
self-intersection $-2p$, we choose a real section $C_{1}\sim C_{0}+(2p+1)f$,
where $f$ is a general real fiber of $\pi_{2p}$, and a pair $(\ell,\overline{\ell})$
of non-real complex conjugate fibers of $\pi_{2p}$. Note that since
the unique intersection point of $C_{1}$ and $C_{0}$ is real, neither
$\ell$ nor $\overline{\ell}$ passes through it. Now we let $\tau\colon V\rightarrow\mathbb{F}_{2p}$
be the smooth real projective surface obtained from $\mathbb{F}_{2p}$
by first blowing-up the pair of non-real complex conjugate points
$q_{1}=C_{1}\cap\ell$ and $\overline{q}_{1}=C_{1}\cap\overline{\ell}$
with respective exceptional divisors $E$ and $\overline{E}$ and
then blowing-up the pair of non-real complex conjugate points $q_{1}'=\ell\cap E$
and $\overline{q}_{1}'=\overline{\ell}\cap\overline{E}$ with respective
exceptional divisors $F$ and $\overline{F}$. The dual graph of $B=C_{0}\cup C_{1}\cup\ell\cup\overline{\ell}\cup E\cup\overline{E}$,
where we identified each curve with its proper transform in $V$,
is a tree depicted in Figure~\ref{fig:hkk} below. 

\begin{figure}[!htb]
\input{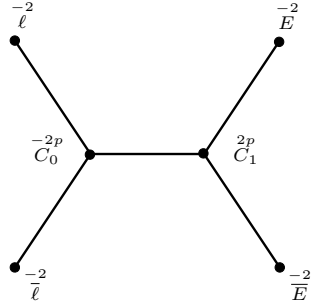} 
\caption{Dual graph of $B$.}  
\label{fig:hkk}
\end{figure}

 Since $B$ is a real curve, $H[-2p,2p]=V\setminus B$ is a smooth
real quasi-projective surface. The abelian group $\mathbb{Z}\langle B\rangle$
is freely generated by the classes of the irreducible curves $C_{0}$,
$C_{1}$, $\ell$, $\overline{\ell}$, $E$ and $\overline{E}$ while
the divisor class group $\mathrm{Cl}(V)$ of $V$ is freely generated
by $C_{0},f,E,\overline{E},F,\overline{F}$.
Using the relations 
\[
\begin{cases}
\tau^{*}C_{1}= & C_{1}+E+\overline{E}+F+\overline{F}\sim C_{0}+(2p+1)f\\
\tau^{*}\ell= & \ell+E+2F\sim f\\
\tau^{*}\overline{\ell}= & \overline{\ell}+\overline{E}+2\overline{F}\sim f
\end{cases}
\]
in $\mathrm{Cl}(V)$, we find that the homomorphism $j_{*}\colon \mathbb{Z}\langle B\rangle\rightarrow\mathrm{Cl}(V)$
induced by the inclusion $B\hookrightarrow V$ is represented by the
matrix 
\[
\left(\begin{array}{cccccc}
1 & 1 & 0 & 0 & 0 & 0\\
0 & 2p+1 & 1 & 1 & 0 & 0\\
0 & -1 & -1 & 0 & 1 & 0\\
0 & -1 & 0 & -1 & 0 & 1\\
0 & -1 & -2 & 0 & 0 & 0\\
0 & -1 & 0 & -2 & 0 & 0
\end{array}\right)
\]
It follows from Theorem \ref{thm:Top-carac} a) and Remark \ref{Rk:Top-carac-rem}
that $S=H[-2p,2p]$ is $\mathbb{Q}$-acyclic, with $H_{1}(S;\mathbb{Z})\simeq\mathbb{Z}_{8p}$.
On the other hand, the real locus of $B$ is homeomorphic to a wedge
of two circles, hence is in particular nonempty. The Galois cohomology
group $H^{2}(G,\mathbb{Z}\langle B\rangle)$ is freely generated
by the classes of $C_{0}$ and $C_{1}$ while $H^{2}(G,\mathrm{Cl}(V))$
is freely generated by the classes of $C_{0}$ and $f$. The relations
above imply that the matrix of the homomorphism $H^{2}(j_{*})\colon H^{2}(G,\mathbb{Z}\langle B\rangle)\rightarrow H^{2}(G,\mathrm{Cl}(V))$
with respect to these bases is the identity. So $H[-2p,2p]$ is a
real plane by virtue of Theorem \ref{thm:Top-carac} b). 
\begin{prop}
For every $p\geq1$, the surface \textup{$H[-2p,2p]$ is a $\mathbb{Q}$-acyclic
fake plane of Kodaira dimension $0$, $\mathbb{R}$-biregularly birationally
equivalent to $\mathbb{C}^{2}$. }\end{prop}
\begin{proof}
The fact that $S=H[-2p,2p]$ if a fake plane follows from the discussion
above. Since $K_{\mathbb{F}_{2p}}\sim-2C_{0}-(2p+2)f$, we deduce
from the ramification formula for $\tau\colon V\rightarrow\mathbb{F}_{2p}$
and the relations in $\mathrm{Cl}(V)$ indicated above that
\begin{eqnarray*}
K_{V}+B & \sim & -2C_{0}-(2p+2)f+E+\overline{E}+2F+2\overline{F}+(C_{0}+C_{1}+E+\overline{E}+\ell+\overline{\ell})\\
  & \sim & f-F-\overline{F},
\end{eqnarray*}
hence that
\[
2(K_{V}+B)\sim2(f-F-\overline{F})\sim(f-2F)+(f-2\overline{F})\sim\ell+E+\overline{\ell}+\overline{E}
\]
So the linear system $|2(K_{V}+B)|$ is nonempty, which implies that
$\kappa(S)\geq0$, and since the intersection matrix of $\ell+E+\overline{\ell}+\overline{E}$ 
is negative definite it follows that $\kappa(S)=0$. That $S$ is
$\mathbb{R}$-biregularly birationally equivalent to $\mathbb{C}^{2}$
can be seen as follows. First since the real morphism $\tau\colon V\rightarrow\mathbb{F}_{2p}$
consists of blow-ups of non-real points only, its restriction to $S$
is a birational morphism $S\rightarrow S'=\mathbb{F}_{2p}\setminus(C_{0}\cup C_{1})$
inducing a diffeomorphism $S(\mathbb{R})\approx S'(\mathbb{R})$.
The smooth real affine surface $S'=\mathbb{F}_{2p}\setminus(C_{0}\cup C_{1})$
admits a real $\mathbb{A}^{1}$-fibration $\rho\colon S'\rightarrow\mathbb{C}$
induced by the restriction of the pencil $\mathbb{F}_{2p}\dashrightarrow\mathbb{CP}^{1}$
generated by the linearly equivalent divisors $C_{1}$ and $C_{0}+(2p+1)F_{0}$,
where $F_{0}\simeq\mathbb{CP}^{1}$ denotes the fiber of $\pi_{2p}$
over the real point $\pi_{2p}(C_{0}\cap C_{1})\in\mathbb{CP}^{1}$.
This fibration has a unique degenerate fiber of multiplicity $2p+1$
consisting of the intersection of $F_{0}$ with $S'$, and so $S'$
is $\mathbb{R}$-biregularly equivalent to $\mathbb{C}^{2}$ by virtue
of Theorem \ref{thm:Real-plane-NegKappa-Rectif}. So $S(\mathbb{R})$
is birationally diffeomorphic to $\mathbb{R}^{2}$. 
\end{proof}

\subsection{The exceptional fake plane $Y(3,3,3)$ (see also \cite[\S\ 5.1.1]{DM15})}

Let $D$ be the union of four general real lines $\ell_{i}\simeq\mathbb{CP}^{1}$,
$i=0,1,2,3$ in $\mathbb{CP}^{2}$ and let $\tau\colon V\rightarrow\mathbb{CP}^{2}$
be the real projective surface obtained by first blowing-up the points
$p_{ij}=\ell_{i}\cap\ell_{j}$ with exceptional divisors $E_{ij}$,
$i,j=1,2,3$, $i\neq j$ and then blowing-up the points $\ell_{1}\cap E_{12}$,
$\ell_{2}\cap E_{23}$ and $\ell_{3}\cap E_{13}$ with respective
exceptional divisors $E_{1}$, $E_{2}$ and $E_{3}$. We let $B=\ell_{0}\cup\ell_{1}\cup\ell_{2}\cup\ell_{3}\cup E_{12}\cup E_{23}\cup E_{13}$.
The dual graphs of $D$, its total transform $\tau^{-1}(D)$ in $V$
and $B$ are depicted in Figure \ref{fig:logkod0-Y333}. 

\begin{figure}[!htb]
\input{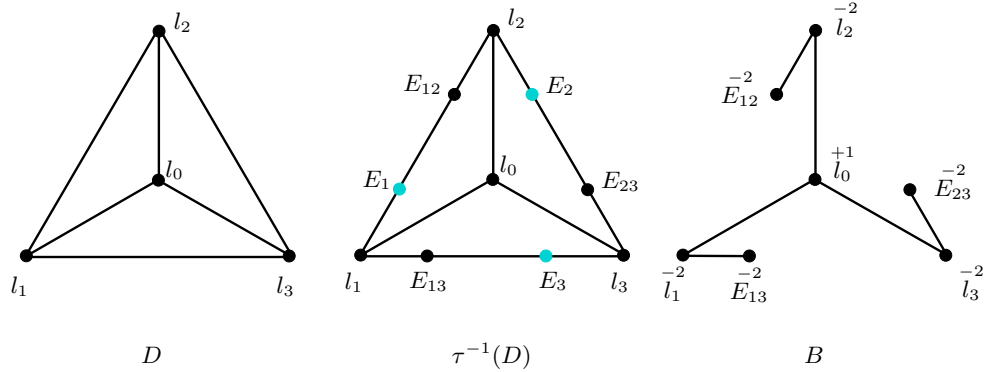} 
\caption{Construction of $Y(3,3,3)$}  
\label{fig:logkod0-Y333}
\end{figure}Since $B$ is a real curve, $Y(3,3,3):=V\setminus B$ is a smooth
real quasi-projective surface. The divisor class group $\mathrm{Cl}(V)$
of $V$ is freely generated by $\ell_{0}$, $E_{12}$, $E_{23}$,
$E_{13}$, $E_{1}$, $E_{2}$ and $E_{3}$ and using the relations
\[
\begin{cases}
\tau^{*}\ell_{1}= & \ell_{1}+E_{12}+E_{13}+2E_{1}+E_{3}\sim\ell_{0}\\
\tau^{*}\ell_{2}= & \ell_{2}+E_{12}+E_{23}+E_{1}+2E_{2}\sim\ell_{0}\\
\tau^{*}\ell_{3}= & \ell_{3}+E_{13}+E_{23}+E_{2}+2E_{3}\sim\ell_{0}
\end{cases}
\]
in $\mathrm{Cl}(V)$ together with the ramification formula for $\tau$,
we find that 
\[
3(K_{V}+B)\sim3\ell_{0}-3(E_{1}+E_{2}+E_{3})\sim(\ell_{1}+\ell_{2}+\ell_{3})+2(E_{12}+E_{13}+E_{23}).
\]
So $\kappa(Y(3,3,3))\geq0$ and since the intersection matrix of the divisor
on the right hand is negative definite, it follows that $\kappa(Y(3,3,3))=0$.
The abelian group $\mathbb{Z}\langle B\rangle$ is freely generated
by the classes of the irreducible curves $\ell_{i}$, $i=0,\ldots,3$,
$E_{12}$, $E_{23}$ and $E_{13}$ and using the above relations,
we find that the homomorphism $j_{*}\colon \mathbb{Z}\langle B\rangle\rightarrow\mathrm{Cl}(V)$
induced by the inclusion $B\hookrightarrow V$ is represented by the
matrix 
\[
M=\left(\begin{array}{ccccccc}
1 & 1 & 1 & 1 & 0 & 0 & 0\\
0 & -1 & -1 & 0 & 1 & 0 & 0\\
\vdots & 0 & -1 & -1 & 0 & 1 & 0\\
\vdots & -1 & 0 & -1 & \vdots & 0 & 1\\
\vdots & -2 & -1 & 0 & \vdots & \vdots & 0\\
\vdots & 0 & -2 & -1 & \vdots & \vdots & \vdots\\
0 & -1 & 0 & -2 & 0 & 0 & 0
\end{array}\right)
\]
which has determinant $\det M=-9$. Furthermore, since $\tau\colon V\rightarrow\mathbb{CP}^{2}$
consists of blow-ups of real points only, $H^{2}(G,\mathbb{Z}\langle B\rangle)\simeq\mathbb{Z}\langle B\rangle\otimes_{\mathbb{Z}}\mathbb{Z}_{2}$,
$H^{2}(G,\mathrm{Cl}(V))\simeq\mathrm{Cl}(V)\otimes_{\mathbb{Z}}\mathbb{Z}_{2}$
and the homomorphism $H^{2}(j_{*})$ is represented by the reduction
modulo $2$ of $M$. It follows from Theorem~\ref{thm:Top-carac} that $Y(3,3,3)$
is a $\mathbb{Q}$-acyclic fake plane, with $H_{1}(Y(3,3,3);\mathbb{Z})\simeq\mathbb{Z}_{9}$. 
\begin{question}
Is the real locus of $Y(3,3,3)$ birationally diffeomorphic to $\mathbb{R}^{2}$
? 
\end{question}

\section{Fake planes of Kodaira dimension 1}

In this section, we first recall basic results on the classification
of $\mathbb{Z}$-acyclic fake planes of Kodaira dimension~$1$. We
then present examples of contractible and $\mathbb{Q}$-acyclic fake
planes of Kodaira dimension $1$ arising from rational real cuspidal
quartics in $\mathbb{CP}^{2}$. We conclude this section with the
construction of families of $\mathbb{Q}$-acyclic fake planes $\mathbb{R}$-biregularly
birationally equivalent to $\mathbb{C}^{2}$.

\subsection{Elements of classification of $\mathbb{Z}$-acyclic fake planes of
Kodaira dimension $1$}

Smooth complex $\mathbb{Z}$-acyclic surfaces of Kodaira dimension
$1$ have been classified by Gurjar and Miyanishi \cite{GuMi87} and
tom Dieck and Petrie \cite{tDPe90}. We recall from \cite{DM15} the
following construction, which provides the real counter-part of this
classification.

\subsubsection{\noindent \label{sub:Kod1-acyclic-construction}}

Let $D\subset\mathbb{CP}^{2}$ be the union of a collection $E_{0,0},\ldots,E_{n,0}\simeq\mathbb{CP}^{1}$ of $n+1\geq3$ real lines intersecting in a same point
$x$ and a general real line $C_{1}\simeq\mathbb{CP}^{1}$. For
every $i=1,\ldots,n$, we choose a pair of coprime integers $1\leq\mu_{i,-}<\mu_{i,+}$
in such a way that for $v_{-}=^{t}(\mu_{1,-},\ldots,\mu_{n,-})\in\mathcal{M}_{n,1}(\mathbb{Z})$
and $\Delta_{+}=\mathrm{diag}(\mu_{1,+},\ldots,\mu_{n,+})\in\mathcal{M}_{n,n}(\mathbb{Z})$,
the following two conditions are satisfied: 
\begin{eqnarray}
\textrm{a) }\eta=n-1-{\displaystyle \sum_{i=1}^{n}\frac{1}{\mu_{i,+}}}>0 & \textrm{and} & b)\textrm{ The matrix }\mathcal{N}=\left(\begin{array}{cc}
-1 & -1\\
v_{-} & \Delta_{+}
\end{array}\right)\textrm{ belongs to }\mathrm{GL}_{n+1}(\mathbb{Z}).\label{eq:Kod1-Conditions}
\end{eqnarray}
Then we let $\tau\colon V\rightarrow\mathbb{CP}^{2}$ be the
smooth real projective surface obtained by the following blow-up procedure: 

1) We first blow-up $x$ with exceptional divisor $C_{0}\simeq\mathbb{CP}^{1}$.
The resulting surface is isomorphic to the Hirzebruch surface $\pi_{1}\colon \mathbb{F}_{1}\rightarrow\mathbb{CP}^{1}$
with $C_{0}$ as the negative section of $\pi_{1}$, the proper transforms
of $E_{0,0},\ldots,E_{n,0}$ are fibers of $\pi_{1}$ while the strict
transform of $C_{1}$ is a section of $\pi_{1}$ disjoint from $C_{0}$. 

2) Then for every $i=1,\ldots,n$, we perform a sequence of blow-up
of real points, starting with the blow-up of $p_{i}=C_{1}\cap E_{i,0}$
in such a way that the following two conditions are satisfied: a)
the inverse image of $p_{i}$ is a chain of $\mathbb{CP}^{1}$ containing
a unique $(-1)$-curve $A(p_{i})$ and b) the coefficients of
$A(p_{i})$ in the total transform of $C_{1}$ and $E_{i,0}$
are equal to $\mu_{i,-}$ and $\mu_{i,+}$ respectively. 
Recall \cite[Example 1.3]{DM15} that letting $x_{i,-}$ and $x_{i,+}$ be local defining equations 
for $C_{1}$ and $E_{i,0}$ in a Zariski open neighborhood $U_i$ of $p_i$, such a blow-up sequence is obtained as the minimal resolution
of the indeterminacies at $p_i$ of the rational map $x_{i,+}^{\mu_{i,+}}/x_{i,-}^{\mu_{i,-}}:U_i \dashrightarrow\mathbb{CP}^{1}$. We denote by $E_{i,1},\ldots,E_{i,r_{i}-1},E_{i,r_{i}}=A(p_{i})$ the corresponding exceptional divisors. 

3) Finally, we perform a sequence of blow-ups starting with the blow-up
of a real point $p_{0}\in E_{0,0}\setminus(C_{0}\cup C_{1})$, with
exceptional divisor $E_{0,1}\simeq\mathbb{CP}^{1}$ and continuing
with a sequence of $r_{0}-1\geq0$ blow-ups of real points $p_{0,i}\in E_{0,i}\setminus E_{0,i-1}$,
$i=1,\ldots,r_{0}-1$, with exceptional divisors $E_{0,i+1}$. We
let $A(p_{0})=E_{0,r_{0}}$. 

\begin{figure}[ht] 
\centering 
\input{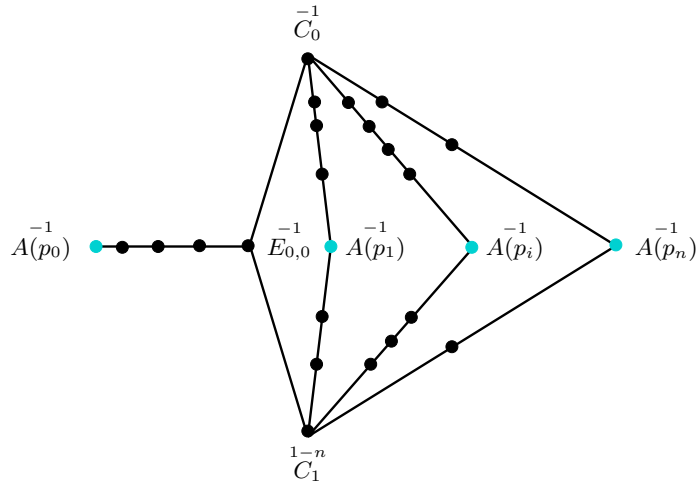} 
\caption{Dual graph of $\tau^{-1}(D)$.}         
\label{fig:logkod1} 
\end{figure} 

The union $B$ of the proper transforms of $C_{0}$,$C_{1}$, and
the divisors $E_{i,j}$, $i=0,\ldots,n$, $j=0,\ldots r_{i-1}$, is
a real subtree of the total transform of $D$ by the
so-constructed morphism $\tau\colon V\rightarrow\mathbb{CP}^{2}$.
By virtue of \cite[Lemma 4.5.3 p. 237]{MiyBook} and the proof of
Theorem 4.6.1 p. 238 in \emph{loc. cit. }we have 
\begin{eqnarray*}
K_{V}+B & \sim & (\pi_{1}\circ\tau)^{*}K_{\mathbb{CP}^{1}}+\ell+\sum_{i=1}^{n}((\pi_{1}\circ\tau)^{*}E_{i,0}-A(p_{i}))\sim(n-1)\ell-\sum_{i=1}^{n}A(p_{i})\sim\eta\ell+N=P+N
\end{eqnarray*}
where $\ell$ is the proper transform of a general real fiber of $\pi_{1}$,
and $N$ is an effective $\mathbb{Q}$-divisor supported on $\sum_{i=1}^{n}((\pi_{1}\circ\tau)^{*}E_{i,0})_{\mathrm{red}}-A(p_{i})$.
So the intersection matrix of $N$ is negative definite and since
${\displaystyle \eta>0}$ by hypothesis, it follows that $P$ is nef.
The surface $S=V\setminus B$ is thus a real surface of Kodaira dimension
$\kappa(S)=1$. Note further that the morphism $\pi_{1}\circ\tau\colon V\rightarrow\mathbb{CP}^{1}$
coincides with that induced by the positive part of $K_{V}+B$ and
that it restricts on $S$ to a fibration $S\rightarrow\mathbb{CP}^{1}$
with general fibers isomorphic to $\mathbb{C}^{*}$. We have the following
classification result:
\begin{thm}
\label{thm:Fake-Kodaira-1} \cite[Theorem 3.3]{DM15} Every $\mathbb{Z}$-acyclic
fake plane $S$ of Kodaira dimension $1$ is isomorphic to a surface
constructed by the above procedure. 
\end{thm}
\noindent We do not know whether any of these surfaces is $\mathbb{R}$-biregularly
birationally equivalent to $\mathbb{C}^{2}$.

\subsection{Fake planes obtained from rational real cuspidal quartic curves }

\subsubsection{\label{sub:OvoideA-Kod1}First construction }

Let $D\subset\mathbb{CP}^{2}$ be a real smooth rational quartic with
a unique cusp $p_{0}$ of multiplicity $3$ and a unique flex $q$
at which $D$ and its tangent $T_{q}$ intersect with multiplicity
$4$. For instance, $D$ is the zero locus of the homogeneous polynomial
$z^{4}-xy^{3}\in\mathbb{C}[x,y,z]$, the corresponding points $p_{0}$
and $q$ being $[1:0:0]$ and $[0:1:0]$. 
 
 Let $\tau\colon V\rightarrow\mathbb{CP}^{2}$
be the blow-up of any real point $p$ of $D$ different from $p_{0}$
and $q$, say with exceptional divisor $E_{0,1}$. Let $E_{0,0}$
and $E_{1,0}$ be the proper transforms of $D$ and $T_{q}$ respectively
in $V$ and let $B=E_{0,0}\cup E_{1,0}$. Then $S=V\setminus B$ is a contractible
fake plane of Kodaira dimension $1$. Indeed, the fact that $S$ is
a $\mathbb{Z}$-acyclic fake plane of Kodaira dimension $1$ can be
deduced directly from Theorem~\ref{thm:Fake-Kodaira-1} by comparing the minimal
log-resolution $\beta\colon V'\rightarrow V$ of the pair $(V,B)$ depicted in Figure~\ref{fig:logkod1-dual1} 
below with Figure \ref{fig:logkod1}  above for $n=2$, the curves
$A(p_{1})$ and $A(p_{2})$ corresponding in this case to
the proper transforms of the tangent line $T_{p_{0}}$ to $B$ at
$p_{0}$ and of the line $\ell_{p_{0}q}$ passing through $p_{0}$
and $q$. 

Alternatively, the $\mathbb{Z}$-acyclicity follows from Theorem \ref{Rk:Top-carac-rem}
via a computation similar to those in the previous section while the
fact that $S(\mathbb{R})\approx\mathbb{R}^{2}$ can be seen more directly
as follows. Since $\tau\colon V\rightarrow\mathbb{CP}^{2}$ consists of
the blow-up of a real point, $V(\mathbb{R})$ is a Klein bottle which
we view as circle bundle $\theta\colon V(\mathbb{R})\rightarrow S^{1}$
with fibers equal to the real loci of the lines through $p$ in $\mathbb{CP}^{2}$.
The sets $E_{0,1}(\mathbb{R})$ and $E_{1,0}(\mathbb{R})$ are two
sections of $\theta$ which do no intersect each other. On the other
hand \emph{$E_{0,0}(\mathbb{R})$} is a connected closed curve which
intersects $E_{0,1}(\mathbb{R})$ and $E_{1,0}(\mathbb{R})$ transversally
in one point and at the point $q$ with multiplicity $4$. It
follows that the homology classes of $E_{0,0}(\mathbb{R})$ and $E_{1,0}(\mathbb{R})$
in $H_{1}(V(\mathbb{R});\mathbb{Z}_{2})$ form a basis of this group.
Since $B(\mathbb{R})$ is not empty, the long exact sequence of relative
homology for the pair $(V(\mathbb{R}),B(\mathbb{R}))$ with $\mathbb{Z}_{2}$-coefficients
implies that $S(\mathbb{R})$ is connected and $\mathbb{Z}_{2}$-acyclic,
hence diffeomorphic to $\mathbb{R}^{2}$. 

\begin{figure}[!htb]
\input{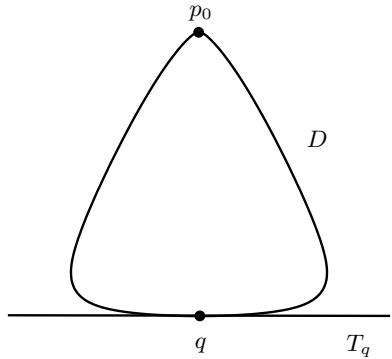} 
\caption{Real ovoid quartic.}
\label{fig:mediator-quartic}
\end{figure}

\begin{figure}[!htb]
\input{logkod1-dual1-f.tex} 
\caption{Dual graph of the total transform of $B\cup T_{p_{0}} \cup \ell_{p_{0}q}\cup E_{0,1}$ in a
minimal real log-resolution $\beta\colon V'\rightarrow V$ of
the pair $(V,B)$.}
\label{fig:logkod1-dual1}
\end{figure}

The contractibility of $S$ follows from the general classification
of such surfaces given in \cite{tDPe90}. It can also be seen directly
as follows: since it is $\mathbb{Z}$-acyclic, it follows from Hurewicz
theorem that $S$ is contractible if and only if it is simply connected.
Since $S\setminus E_{0,1}\simeq\mathbb{CP}^{2}\setminus(D\cup T_{q})$
is the complement of a closed submanifold of real codimension $2$
in $S$, $\pi_{1}(S)$ is a quotient of $\pi_{1}(\mathbb{CP}^{2}\setminus(D\cup T_{q}))$.
Furthermore, by virtue of \cite[Lemma 2.3(a)]{Zai00}, we have an
exact sequence $0\rightarrow\mathbb{Z\cdot\langle\mu}\rangle\rightarrow\pi_{1}(\mathbb{CP}^{2}\setminus(D\cup T_{q}))\rightarrow\pi_{1}(\mathbb{CP}^{2}\setminus D)\rightarrow0$
where $\mathbb{Z\cdot\langle\mu}\rangle$ is the free abelian group
generated by a meridian of the line $T_{q}$. Since $\pi_{1}(\mathbb{CP}^{2}\setminus D)$
is abelian \cite[Proposition 4.3 p.~130]{Dimca}, so is $\pi_{1}(\mathbb{CP}^{2}\setminus(D\cup T_{q}))$,
implying in turn that $\pi_{1}(S)$ is abelian, hence trivial.

\subsubsection{\label{sub:OvoideB-kod1}Second construction}

Let $D\subset\mathbb{CP}^{2}$ be a real smooth rational quartic with
a unique cusp $p_{0}$ of multiplicity $3$ and a pair of real flexes
$q_{1}$ and $q_{2}$ at which $D$ and its tangent $T_{q_{i}}$ intersect
with multiplicity $3$. For instance, $D$ is the zero locus of the
homogeneous polynomial $x^{3}y-z^{3}x+z^{4}\in\mathbb{C}[x,y,z]$,
the corresponding points $p_{0}$ and $q_{1}$ and $q_{2}$ being
$[0:1:0]$, $[1:0:0]$ and $[\frac{3}{2}:\frac{4}{27}:1]$.
Note that since a complex quartic with these properties 
is unique up to complex projective equivalence (see e.g. \cite[pp. 135, 146]{Na83}), 
the construction below applies verbatim to any other real model of $D$. 
 Let $\tau\colon V=\mathbb{F}_{1}\rightarrow\mathbb{CP}^{2}$
be the blow-up of the intersection point $p$ of $T_{q_{1}}$ with
$D$ different from $q_{1}$, say with exceptional divisor $F_{1}$. Since
$p$ is real, $\tau$ is a real morphism and hence, the complement
in $V$ of the proper transform $B$ of $D\cup T_{q_{1}}$ is a real surface.
The relations $\tau^{*}D=D+F_{1}$, $\tau^{*}T_{q_{1}}=T_{q_{1}}+F_{1}$ in $\mathrm{Cl}(V)$ and
$D\sim4T_{q_{1}}$  in $\mathrm{Cl}(\mathbb{CP}^{2})$
imply that the homomorphism $j_{*}\colon \mathbb{Z}\langle B\rangle\rightarrow\mathrm{Cl}(V)$
induced by the inclusion $j\colon B\hookrightarrow V$ is represented in
appropriate bases by a matrix of the form 
\[
M=\left(\begin{array}{cc}
4 & 1\\
-1 & -1
\end{array}\right).
\]
Furthermore, $H^{2}(G,\mathbb{Z}\langle B\rangle)\simeq\mathbb{Z}\langle B\rangle\otimes_{\mathbb{Z}}\mathbb{Z}_{2}$,
$H^{2}(G,\mathrm{Cl}(V))\simeq\mathrm{Cl}(V)\otimes_{\mathbb{Z}}\mathbb{Z}_{2}$
as $\tau$ is the blow-up of a real point, and the homomorphism $H^{2}(j_{*})$
is represented by the reduction modulo $2$ of $M$. It follows from
Theorem~\ref{thm:Top-carac} and Remark \ref{Rk:Top-carac-rem} that $S$
is a $\mathbb{Q}$-acyclic fake plane, with $H_{1}(S;\mathbb{Z})\simeq\mathbb{Z}_{3}$ as $\det(M)=-3$.
In contrast with the $\mathbb{Z}$-acyclic fake planes considered
in $\S$ \ref{sub:Kod1-acyclic-construction} which all admit a log-canonical
real fibration over $\mathbb{CP}^{1}$ with general fibers isomorphic
to $\mathbb{C}^{*}$, the surface $S$ just constructed admits a fibration
of this type over the complex line $\mathbb{C}$. Namely, letting
$\ell_{p_{0}q_{1}}=\{z=0\}$ and $\ell_{p_{0}p}=\{z-x=0\}$ be the lines through $p_{0}$, $q_{1}$ and the line through $p_{0}$, $p$ respectively,
the divisors $D$, $3T_{p_{0}}+T_{q_{1}}=\{x^3y=0\}$ and $3\ell_{p_{0}q_{1}}+\ell_{p_{0}p}=\{z^3(z-x)=0\}$
in $\mathbb{CP}^{2}$ generate a real pencil $V\dashrightarrow\mathbb{CP}^{1}$ containing $D=\{x^3y+z^3(z-x)=0\}$ as a member and
whose lift in the minimal real log-resolution $\beta\colon \tilde{V}\rightarrow V$
of the pair $(V,B)$ coincides with the $\mathbb{P}^{1}$-fibration
$\tilde{V}\rightarrow\mathbb{CP}^{1}$ generated by the proper transform
of $D$ (see Figure~\ref{fig:logkod1-2} below). This fibration restricts on $S$ to
a real morphism $S\rightarrow\mathbb{C}$ with general fibers isomorphic
to $\mathbb{C}^{*}$ and two degenerate fibers: one consisting of
the disjoint union of $\ell_{p_{0}p}\cap S\simeq\mathbb{C}$ and $\ell_{p_{0}q_{1}}\cap S\simeq\mathbb{C}^{*}$
the second component having multiplicity $3$, and a second one consisting
of the intersection $T_{p_{0}}\cap S=\mathbb{C}^{*}$, also with multiplicity
$3$. The fact that $\kappa(S)=1$ then follows from a similar computation
as in $\S$ \ref{sub:Kod1-acyclic-construction}, see e.g. \cite[Theorem 4.6.2 p. 237]{MiyBook},
which also renders the conclusion that $\tilde{V}\rightarrow\mathbb{CP}^{1}$
coincides with the fibration induced by the positive part of $K_{\tilde{V}}+\beta^{-1}(B)$. 

\begin{figure}[ht] 
\centering 
\input{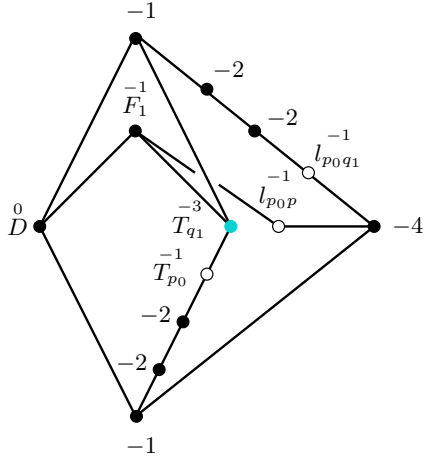} 
\caption{Dual graph of $\beta^{-1}(B\cup T_{p_{0}}\cup\ell_{p_{0}q_{1}}\cup\ell_{p_{0}p})$.}         
\label{fig:logkod1-2} 
\end{figure} 

\subsection{$\mathbb{Q}$-acyclic fake planes with real loci birationally diffeomorphic
to $\mathbb{R}^{2}$ (see also \cite[Remark 3.2]{DM15}).}

To construct families of $\mathbb{Q}$-acyclic fake planes $\mathbb{R}$-biregularly
birationally isomorphic to $\mathbb{C}^{2}$, we start similarly as in
$\S$ \ref{sub:Kod1-acyclic-construction} above with the union $D\subset\mathbb{CP}^{2}$
of a real line $E_{0,0}\simeq\mathbb{CP}^{1}$, a collection $(E_{1,0},\overline{E}_{1,0}),\ldots,(E_{m,0},\overline{E}_{m,0})$
of $m\geq1$ distinct pairs of non-real complex conjugated lines all
meeting in a same real point $x$, and of a general real line $C_{1}\simeq\mathbb{CP}^{1}$.
For every $i=1,\ldots,m$, we choose a pair of coprime integers $1\leq\nu_{i,-}<\nu_{i,+}$
in such a way that for 
\[
\begin{cases}
v_{-}= & ^{t}(\nu_{1,-},\ldots,\nu_{m,-},\nu_{1,-},\ldots,\nu_{m,-})\in\mathcal{M}_{2m,1}(\mathbb{Z})\\
\Delta_{+}= & \mathrm{diag}(\nu_{1,+},\ldots,\nu_{m,+},\nu_{1,+},\ldots,\nu_{m,+})\in\mathcal{M}_{2m,2m}(\mathbb{Z}),
\end{cases}
\]
the following two conditions are satisfied: 

\begin{equation}
\textrm{a') }\eta=2m-1{\displaystyle -2\sum_{k=1}^{m}\frac{1}{\nu_{k,+}}}>0  \quad\textrm{and}\quad b')\textrm{ The matrix }\mathcal{M}=\left(\begin{array}{cc}
-1 & -1\\
v_{-} & \Delta_{+}
\end{array}\right)\textrm{ belongs to }\mathrm{GL}_{2m+1}(\mathbb{Q}).\label{eq:Kod1-Conditions-rectif}
\end{equation}

Then we let $\tau\colon V\rightarrow\mathbb{CP}^{2}$ be the
real smooth projective surface obtained by the following blow-up procedure: 

1) We first blow-up $x$ with exceptional divisor $C_{0}\simeq\mathbb{CP}^{1}$.
The resulting surface is isomorphic the Hirzebruch surface $\pi_{1}\colon \mathbb{F}_{1}\rightarrow\mathbb{CP}^{1}$
with $C_{0}$ as the negative section of $\pi_{1}$, the proper transform
of $C_{1}$ is a section of $\pi_{1}$ disjoint from $C_{0}$, the
proper transform of $E_{0,0}$ is a real fiber of $\pi_{1}$ while
the proper transforms of the $E_{i,0}$ and $\overline{E}_{i,0}$,
$i=1,\ldots,m$ are pairs of non-real complex conjugate fibers of
$\pi_{1}$. 

2) Then for every $i=1,\ldots,m$, we perform a sequence of blow-ups
of pairs of non-real complex conjugate points, starting with the blow-up
of the points $p_{i}=C_{1}\cap E_{i,0}$ and $\overline{p}_{i}=C_{1}\cap\overline{E}_{i,0}$
in such a way that the following two conditions are satisfied: a)
the inverse images of $p_{i}$ and $\overline{p}_{i}$ are complex
conjugate chains of $\mathbb{CP}^{1}$ each containing a unique $(-1)$-curve
$A(p_{i})$ and $A(\overline{p}_{i})$ respectively and b)
the coefficients of $A(p_{i})$ (resp. of $A(\overline{p}_{i})$
in the total transform of $C_{1}$ and $E_{i,0}$ (resp. $\overline{E}_{i,0}$)
are equal to $\nu_{i,-}$ and $\nu_{i,+}$ respectively. We denote
by $E_{i,1},\ldots,E_{i,r_{i}-1},E_{i,r_{i}}=A(p_{i})$ and $\overline{E}_{i,1},\ldots,\overline{E}_{i,r_{i}-1},\overline{E}_{i,r_{i}}=A(\overline{p}_{i})$ the corresponding exceptional divisors. 

3) Finally, we perform a sequence of blow-ups starting with the blow-up
of a real point $p_{0}\in E_{0,0}\setminus(C_{0}\cup C_{1})$, with
exceptional divisor $E_{0,1}\simeq\mathbb{CP}^{1}$ and continuing
with a sequence of $r_{0}-1\geq0$ blow-ups of real points $p_{0,i}\in E_{0,i}\setminus E_{0,i-1}$,
$i=1,\ldots,r_{0}-1$, with exceptional divisors $E_{0,i+1}$. We
let $A(p_{0})=E_{0,r_{0}}$. 

\begin{figure}[ht] 
\centering       
\input{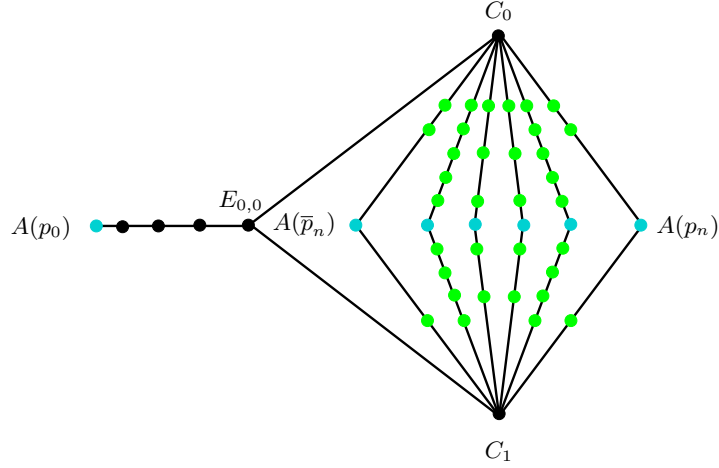} 
\caption{Dual graph of $\tau^{-1}(D)$.}
 \label{fig:logkod1-conjugate} 
\end{figure} 

The union $B$ of the proper transforms of $C_{0}$,$C_{1}$, and
the divisors $E_{i,j}$ and $\overline{E}_{i,j}$, $i=0,\ldots,m$,
$j=0,\ldots r_{i-1}$, is a real subtree of the total transform $\tilde{B}$
of $D$ by the so-constructed morphism $\tau\colon V\rightarrow\mathbb{CP}^{2}$.
Condition a') guarantees by the same argument as in $\S$\ref{sub:Kod1-acyclic-construction}
that $S=V\setminus B$ is a real surface of Kodaira dimension $1$.
The classes of the curves $C_{0}$, $E_{i,j}$ and $\overline{E}_{i,j}$,
$i=1,\ldots,m$, $j=0,\ldots r_{i-1}$, $E_{0,0}$, $C_{1}-E_{0,0}$,
$E_{0,i}-E_{0,0}$, $\overline{E}_{0,i}-E_{0,0}$, $i=1,\ldots,m$,
form a basis of the free abelian group $\mathbb{Z}\langle B\rangle$
generated by the irreducible components of $B$ while $\mathrm{Cl}(V)$
is freely generated by the classes of $C_{0}$, $E_{i,j}$ and $\overline{E}_{i,j}$,
$i=1,\ldots,m$, $j=0,\ldots r_{i-1}$, $E_{0,0}$, $A(p_{0})$,
$A(p_{i})$ and $A(\overline{p}_{i})$, $i=1,\ldots,m$ .
In these bases, the matrix of the homomorphism $j_{*}\colon \mathbb{Z}\langle B\rangle\rightarrow\mathrm{Cl}(V)$
induced by the inclusion takes the form 
\[
M=\left(\begin{array}{cc}
\mathrm{id_{n}} & *\\
0 & \mathcal{M}
\end{array}\right)
\]
where $n=2(\sum_{i=0}^{m}r_{i-1}+1)$. Since $\mathcal{M\in}\mathrm{GL}_{2m+1}(\mathbb{Q})$
by hypothesis, we deduce from Theorem \ref{thm:Top-carac} a) that
$S$ is $\mathbb{Q}$-acyclic. But it is never $\mathbb{Z}$-acyclic
as $\nu_{i,+}\geq2$ for every $i=1,\ldots,m$. On the other hand,
the real locus of $B$ is nonempty and the induced homomorphism $H^{2}(j_{*})\colon H^{2}(G;\mathbb{Z}\langle B\rangle)\rightarrow H^{2}(G;\mathrm{Cl}(V))$
is an isomorphism, implying that $S(\mathbb{R})\approx\mathbb{R}^{2}$
by virtue of b) in Theorem \ref{thm:Top-carac}. 
\begin{prop}
Every real surface $S$ as above\textup{\emph{ is a $\mathbb{Q}$-acyclic
fake plane of Kodaira dimension $\mbox{1}$, $\mathbb{R}$-biregularly
birationally equivalent to $\mathbb{C}^{2}$.}}
\end{prop}

\begin{proof}
It remains to show that $S(\mathbb{R})$ is birationally diffeomorphic
to $\mathbb{R}^{2}$. The real birational morphism $\tau\colon V\rightarrow\mathbb{CP}^{2}$
factors through the real projective surface $\tau'\colon V'\rightarrow\mathbb{CP}^{2}$
obtained from $\mathbb{CP}^{2}$ by blowing-up $x$ with exceptional
divisor $C_{0}$ and then only performing the last sequence of blow-ups
described in 3) above, the one starting with the blow-up of a real
point $p_{0}\in E_{0,0}\setminus(C_{0}\cup C_{1})$. Since the induced
real birational morphism $\beta\colon V\rightarrow V'$ consists of blow-ups
of pairs of non-real complex conjugate points only, it restrict to
a birational diffeomorphism between the real loci of $V$ and $V'$ respectively,
mapping $S(\mathbb{R})$ diffeomorphically onto the real loci of the
surface $S'=V'\setminus B'$ where $B'=C_{0}\cup C_{1}\cup\bigcup_{j=0}^{r_{0}-1}E_{0,j}$.
A similar counting argument as in \ref{sec:DiffBir} above shows that the surface obtained from $V'$ by first blowing-up the real point
$C_{1}\cap E_{0,0}$ with exceptional divisor $C$ and then contracting
successively the proper transforms of $C_{0}$, $E_{0,0}$, $E_{0,1},\ldots,E_{0,r_{0}-1}$
has Picard rank $2$, hence is a Hirzebruch surface $\pi_{n}\colon \mathbb{F}_{n}\rightarrow\mathbb{CP}^{1}$
for some $n\geq0$, in which the proper transforms of $C_{1}$ and
$C$ are respectively a real fiber of $\pi_{n}$ and a real section
of it, with self-intersection $r_{0}$. The so-constructed real birational
map $\theta\colon V'\dashrightarrow\mathbb{F}_{n}$ induces an isomorphism
between the real loci of $S'$ and that of $\mathbb{F}_{n}\setminus(C\cup C_{1})$,
and the composition $\theta\circ\beta\mid_{S}\colon S\rightarrow\mathbb{F}_{n}\setminus(C\cup C_{1})\simeq\mathbb{C}^{2}$
is the desired $\mathbb{R}$-biregular birational isomorphism. 

\begin{figure}[ht] 
\centering       
\input{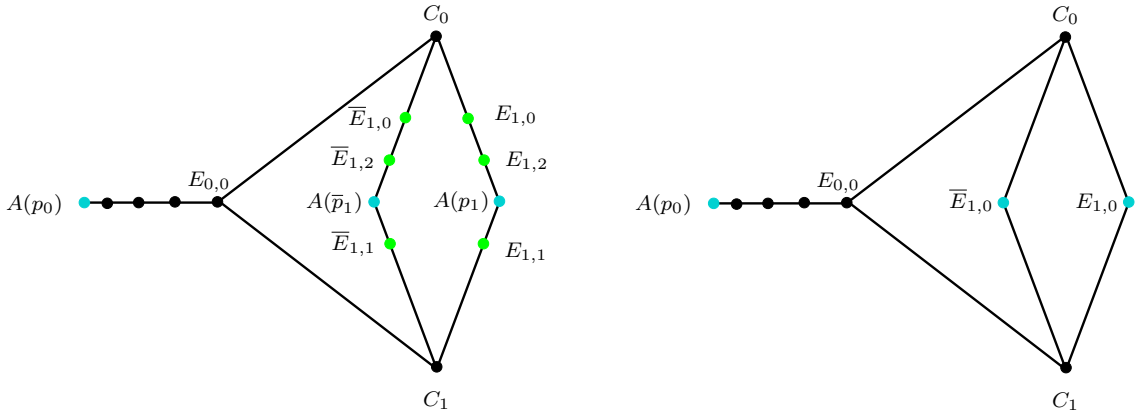} 
\caption{Dual graphs of $\tau^{-1}(D)$ and ${\tau'}^{-1}(D)$ in the case $m=1$, $r_1=3$, $(\nu_{1,-},\nu_{1,+})=(2,3)$.}
 \label{fig:logkod1-rectif} 
\end{figure} 

\end{proof}

\section{Fake planes of general type}

By virtue of \cite{Wak78}, the complement $S$ of an irreducible
rational cuspidal curve $B\subset\mathbb{C}\mathbb{P}^{2}$ is a smooth
affine surface of general type as soon as $B$ has a least three cusps.
Such a surface $S$ is $\mathbb{Q}$-acyclic by virtue of Theorem
\ref{thm:Top-carac} a), with $H_{1}(S;\mathbb{Z})\simeq\mathbb{Z}_{d}$,
where $d\geq3$ denotes the degree of $B$. If in addition $B$ is
a real curve with non empty real locus, then Theorem \ref{thm:Top-carac}
b) implies that $S$ is a real plane if and only if $d$ is odd. This
can be seen more directly as follows: when non empty, $B(\mathbb{R})$
is homeomorphic to a simple closed curve in $\mathbb{R}\mathbb{P}^{2}$,
whose homotopy class is either trivial if $d$ is even or equal to
a generator of $\pi_{1}(\mathbb{RP}^{2})$ if $d$ is odd. So $S(\mathbb{R})=\mathbb{RP}^{2}\setminus B(\mathbb{R})$
is either diffeomorphic to the disjoint union of $\mathbb{R}^{2}$
with an open Möebius band in the first case or to $\mathbb{R}^{2}$
in the second case. 

In this section, we first present different examples of contractible
fake planes constructed from arrangements of lines and rational real
cuspidal quartics in $\mathbb{CP}^{2}$ by the so-called cutting-cycle
construction of tom Dieck and Petrie \cite{tDPe93}. Then we construct
an infinite family of $\mathbb{Z}$-acyclic fake planes of general
type whose real loci are all birationally diffeomorphic to $\mathbb{R}^{2}$.

\subsection{Fake planes of general type obtained from rational real cuspidal
quartic curves }

\subsubsection{\label{sub:OvoideB-kod2} First construction}

We consider again as in $\S$ \ref{sub:OvoideB-kod1} a real smooth
rational quartic $D_{1}\subset\mathbb{CP}^{2}$ with a unique cusp
$p_{1,0}$ of multiplicity $3$ and a pair of real flexes $q_{1,1}$
and $q_{1,2}$ at which $D_{1}$ and its tangent $T_{q_{i,1}}(D_{1})$
intersect with multiplicity $3$. We let $\tau_{1}\colon V_{1}\rightarrow\mathbb{CP}^{2}$
be the real projective surface obtained from $\mathbb{CP}^{2}$ by
first blowing-up the intersection point $p_{1}$ of $T_{q_{1,1}}(D_{1})$
with $D_{1}$ different from $q_{1,1}$ with exceptional divisor $F_{1,1}$,
then blowing-up the intersection point of $F_{1,1}$ with the proper
transform of $D_{1}$ with exceptional divisor $F_{1,2}$ and then
blowing-up the intersection point of $F_{1,2}$ with the proper transform
of $D_{1}$, with exceptional divisor $A(p_{1})$. The complement
in $V_{1}$ of $B_{1}=D_{1}\cup T_{q_{1,1}}(D_{1})\cup F_{1,1}\cup F_{1,2}$
is a smooth $\mathbb{Z}$-acyclic real surface $S_{1}$ with $S_{1}(\mathbb{R})\approx\mathbb{R}^{2}$. 

\begin{figure}[ht] 
\centering 
\input{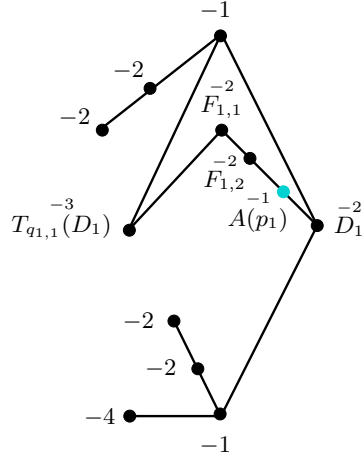} 
\caption{Dual graph of the total transform of $B_{1}\cup A(p_{1})$ in a
minimal real log-resolution $\beta_{1}\colon V_{1}'\rightarrow V_{1}$ of
the pair $(V_{1},B_{1})$.}         
\label{fig:logkod2-1} 
\end{figure} 

Indeed, by construction, 
\begin{eqnarray*}
\tau_{1}^{*}D_{1} & = & D_{1}+F_{1,1}+2F_{1,2}+3A(p_{1})\\
\tau_{1}^{*}T_{q_{1,1}}(D_{1}) & = & T_{q_{1,1}}(D_{1})+F_{1,1}+F_{1,2}+A(p_{1})
\end{eqnarray*}
so that using the relation $\tau_{1}^{*}D_{1}\sim4\tau_{1}^{*}T_{q_{1,1}}(D_{1})$
in $\mathrm{Cl}(V_{1})$, we obtain that in the bases of $\mathbb{Z}\langle B_{1}\rangle$
and $\mathrm{Cl}(V_{1})$ given by the classes of the real
curves, $T_{q_{1,1}}(D_{1})$, $F_{1,1}$, $F_{1,2}$, $D_{1}$ and
$T_{q_{1,1}}(D_{1})$, $F_{1,1}$, $F_{1,2}$, $A(p_{1})$ respectively,
the homomorphism $j_{*}\colon \mathbb{Z}\langle B_{1}\rangle\rightarrow\mathrm{Cl}(V_{1})$
induced by the inclusion $B_{1}\hookrightarrow V_{1}$ is represented
by the matrix 
\[
M=\left(\begin{array}{cccc}
1 & 0 & 0 & 4\\
0 & 1 & 0 & 3\\
0 & 0 & 1 & 2\\
0 & 0 & 0 & 1
\end{array}\right)\in\mathrm{GL}_{4}(\mathbb{Z}).
\]
So $S_{1}$ is $\mathbb{Z}$-acyclic by Theorem \ref{thm:Top-carac}
a), and the same type of arguments as in $\S$ \ref{sub:OvoideA-Kod1}
implies that $S_{1}$ is even contractible. On the other hand, since
$\tau_{1}\colon V_{1}\rightarrow\mathbb{CP}^{2}$ consists of blow-ups of
real points only, $H^{2}(j_{*})$ is represented by the reduction
modulo $2$ of $M$, which is again invertible. Since $B_{1}(\mathbb{R})$
is nonempty, $S_{1}$ is a real plane by Theorem \ref{thm:Top-carac}
b).

\subsubsection{\label{sub:Ramphoid-kod2} Second construction}

We let $D_{2}\subset\mathbb{CP}^{2}$ be a real rational ramphoid
quartic, i.e. with a unique cusp $p_{2,0}$ with multiplicity sequence
$(2,2,2)$, and three real flexes $q_{2,1}$, $q_{2,2}$, $q_{2,3}$
at which $D_{2}$ intersects its tangent $T_{q_{2,i}}(D_{2})$ with
multiplicity $3$. We let $\tau_{2}\colon V_{2}\rightarrow\mathbb{CP}^{2}$
be the real birational morphism obtained by blowing-up the intersection
point $p_{2}$ of $D_{2}$ with $T_{q_{2,1}}(D_{2})$ distinct from
$q_{2,1}$, with exceptional divisor $F_{2,1}$, then blowing-up the
intersection point of $F_{2,1}$ with the proper transform of $D_{2}$,
with exceptional divisor $F_{2,2}$ and then blowing-up the intersection
point of $F_{2,2}$ with the proper transform of $D_{2}$, with exceptional
divisor $A(p_{2})$. We let $S_{2}$ be the smooth real surface obtained
as the complement of $B_{2}=D_{2}\cup T_{q_{2,1}}(D_{2})\cup F_{2,1}\cup F_{2,2}$
in $V_{2}$. 

\begin{figure}[ht] 
\centering 
\input{logkod2-dual2-f.tex} 
\caption{Dual graph of the total transform of $B_{2}\cup A(p_{2})$ in a
minimal real log-resolution $\beta_{2}\colon V_{2}'\rightarrow V_{2}$ of
the pair $(V_{2},B_{2})$.}         
\label{fig:logkod2-2} 
\end{figure}

The same computation as in $\S$ \ref{sub:OvoideB-kod2} shows that
$S_{2}$ is a contractible real plane.

\subsubsection{\label{sub:Bicusp-kod2} Third construction}

Here we start with a real rational bi-cuspidal quartic $D_{3}\subset\mathbb{CP}^{2}$,
with two real cusps $p_{3,\infty}$ and $p_{3,0}$ with respective
multiplicity sequences $(2,2)$ and $(2)$ such that $D_{3}$ intersects
its tangent $T_{p_{3,0}}(D_{3})$ with multiplicity $3$ at $p_{3,0}$ and transversally at another real point $p_3$.
We let $\tau_{3}\colon V_{3}\rightarrow\mathbb{CP}^{2}$ be the smooth real
projective surface obtained from $\mathbb{CP}^{2}$ by first blowing-up
$p_{3}$ with exceptional divisor $F_{3,1}$, then
blowing-up the intersection point of $F_{3,1}$ with the proper transform
of $D_{3}$, with exceptional divisor $F_{3,2}$ and then blowing-up
the intersection point of $F_{3,2}$ with the proper transform of
$D_{3}$, with exceptional divisor $A(p_{3})$. We let $S_{3}$ be
the smooth real surface obtained as the complement of $B_{3}=D_{3}\cup T_{p_{3,0}}(D_{3})\cup F_{3,1}\cup F_{3,2}$
in $V_{3}$. 

\begin{figure}[ht] 
\centering 
\input{logkod2-dual3-f.tex} 
\caption{Dual graph of the total transform of $B_{3}\cup A(p_{3})$ in a
minimal real log-resolution $\beta_{3}\colon V_{3}'\rightarrow V_{3}$ of
the pair $(V_{3},B_{3})$.}         
\label{fig:logkod2-3} 
\end{figure}

Similar arguments as in $\S$ \ref{sub:OvoideB-kod2} imply that $S_{3}$
is a contractible real plane. 
\begin{prop}
\label{prop:three-generaltype-planes} The surfaces $S_{1}$, $S_{2}$
and $S_{3}$ constructed in $\S$ \ref{sub:OvoideB-kod2}, $\S$ \ref{sub:Ramphoid-kod2}
and $\S$ \ref{sub:Bicusp-kod2} above are pairwise non isomorphic
contractible fake planes of general type. \end{prop}
\begin{proof}
Letting $\beta_{i}\colon V_{i}'\rightarrow V_{i}$, $i=1,2,3$, be the minimal
real log-resolutions of the pair $(V_{i},B_{i})$ and $B_{i}'=\beta_{i}^{-1}(B_{i})$,
the pairs $(V_{i}',B_{i}')$ are real minimal SNC-completion of the
$S_{i}$. Note that every irreducible component of $B_{i}'$ has self-intersection
$\leq-1$ and that every $(-1)$-curve in $B_{i}'$ intersect three
other irreducible components of $B_{i}'$ (see Figure \ref{fig:logkod2-1}, \ref{fig:logkod2-2} and \ref{fig:logkod2-3}
above). It follows that every birational map $\varphi\colon (V'',B'')\dashrightarrow(V_{i}',B_{i}')$
from another SNC-completion $(V'',B'')$ of $S_{i}$ restricting to
an isomorphism between $V''\setminus B''$ and $V_{i}'\setminus B_{i}'$
is a morphism. Now if $S_{i}$ was isomorphic to $S_{j}$ for some
$j\neq i$, then the birational map $\varphi\colon (V_{j}',B_{j}')\dashrightarrow(V_{i}',B_{i}')$
extending an isomorphism $S_{j}\stackrel{\sim}{\rightarrow}S_{i}$
would be an isomorphism mapping $B_{j}'$ isomorphically onto $B_{i}'$.
In particular, the weighted dual graphs of $B_{j}'$ and $B_{i}'$
would be isomorphic, which is not the case. So $S_{1}$, $S_{2}$
and $S_{3}$ are pairwise non isomorphic. It remains to show that
they are all of general type. Since $\mathbb{C}^{2}$ is the only
smooth $\mathbb{Z}$-acyclic real plane and since there are no $\mathbb{Z}$-acyclic
real plane of Kodaira dimension $0$, to show that $\kappa(S_{i})=2$,
it is enough to check that $S_{i}$ is isomorphic neither to $\mathbb{C}^{2}$
nor to one of the surfaces described in $\S$~\ref{sub:Kod1-acyclic-construction}.
If $S_{i}$ was isomorphic to $\mathbb{C}^{2}$, then we would have
a morphism $\varphi\colon (\mathbb{CP}^{2},\ell)\rightarrow(V',B')$, where
$\ell\simeq\mathbb{CP}^{1}$ is a real line, restricting to an isomorphism
between $\mathbb{CP}^{2}\setminus\ell$ and $S_{i}$, which is impossible
as $\ell^{2}>0$. Now suppose that $S_{i}$ is isomorphic to a surface
obtained by the blow-up procedure $\tau''\colon V''\rightarrow\mathbb{CP}^{2}$
described in $\S$ \ref{sub:Kod1-acyclic-construction} and let $(V'',B'')$
be the corresponding SNC-completion with boundary $B''$ consisting
of the proper transforms of $C_{0}$,$C_{1}$, and the divisors
$E_{i,j}$, $i=0,\ldots,n$, $j=0,\ldots r_{i-1}$. By construction,
all irreducible components of $B''$ have self-intersection $\leq-1$
and $B''$ contains at most three $(-1)$-curves: the proper transform
of $C_{0}$ and $E_{0,0}$, and the proper transform of $C_{1}$ if
$n=2$, each of them intersecting at least three other irreducible
components of $B''$, except in the case where $r_{0}=1$ . Since
the birational map $\varphi\colon (V'',B'')\dashrightarrow(V_{i}',B_{i}')$
induced by the isomorphism $V''\setminus B''\simeq S_{i}\simeq V'\setminus B$
is a morphism, the structure of $B''$ implies that $\varphi$ is
in fact an isomorphism of pairs, except possibly when $r_{0}=1$
where it can consists of the contraction of $E_{0,0}$ followed by
an isomorphism of pairs. In both cases, we reach a contradiction by
comparing the weighted dual graph of $B_{i}'$ and the one $B''$
or its image by the contraction of $E_{0,0}$ in the case where $r_{0}=1$.
So $S_{i}$ is a fake plane of general type. \end{proof}
\begin{question}
Is the real locus of $S_{1}$, $S_{2}$ and $S_{3}$ birationally
diffeomorphic to $\mathbb{R}^{2}$ ? 
\end{question}

\subsection{$\mathbb{Z}$-acyclic fake planes of general type $\mathbb{R}$-biregularly
birationally equivalent to $\mathbb{C}^{2}$}

The projective dual $D\subset\mathbb{CP}^{2}$ of the nodal cubic
$C$ with equation $(x-y)(x^{2}+y^{2})-xyz=0$ is a real rational
quartic with three ordinary cusps: a real one $p_{0}$ corresponding
to the real flex $[1:1:0]$ of $C$ and a pair of non-real complex
conjugate ones $q$ and $\overline{q}$ corresponding to the pair of
non-real conjugate flexes $[i:1:0]$ and $[-i:1:0]$ of $C$. The
tangent $T_{p_{0}}(D)$ of $D$ at $p_{0}$ intersects $D$ with multiplicity
$3$ at $p_{0}$ and transversally at another real point $p$. Let
$(\mu,\nu)$ be a pair of positive integers such that $4\nu-\mu=\pm1$
and let $\tau\colon V\rightarrow\mathbb{CP}^{2}$ be the real birational
morphism obtained by first blowing-up $p$ with exceptional divisor
$E_{1}$ and then, similarly as in \ref{sub:Kod1-acyclic-construction} 2), blowing-up a sequence of real points on the successive
total transforms of $E_{1}$ in such a way that the following two
conditions are satisfied: a) the inverse image of $p$ is a chain
of curves isomorphic to $\mathbb{CP}^{1}$ containing a unique $(-1)$-curve
$A(p)$ and b) the coefficients of $A(p)$ in the total transform
of $D$ and $T_{p_{0}}(D)$ are equal to $\mu$ and $\nu$ respectively.
We denote the corresponding exceptional divisors by $E_{1},\ldots,E_{r-1},E_{r}=A(p)$
and we let $B=D\cup T_{p_{0}}(D)\cup\bigcup_{i=1}^{r-1}E_{i}$. The
weighted dual graph of the total transform of $B\cup A(p)$ in a real
minimal log-resolution $\beta\colon V'\rightarrow V$ of the pair $(V,B)$
is depicted in Figure \ref{fig:logkod2-tadam}.

\begin{figure}[!htb]
\input{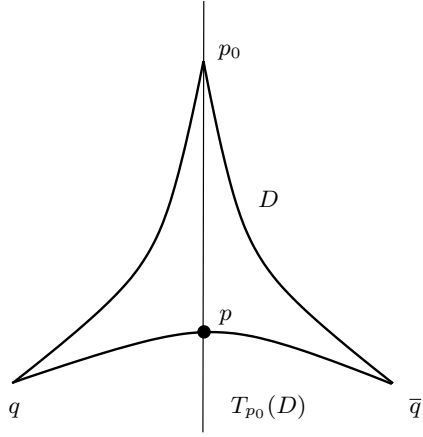} 
\caption{Real tricuspidal quartic}  
\label{fig:tricusp-quartic}
\end{figure}

\begin{figure}[!htb]
\input{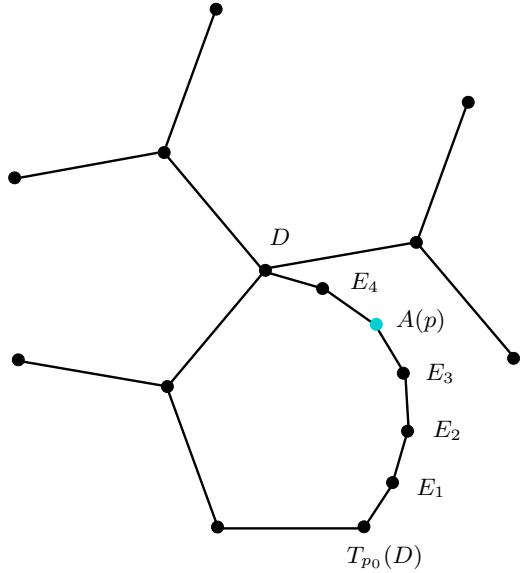} 
\caption{Dual graph of the total transform of $B\cup A(p)$ in a
minimal real log-resolution $\beta \colon V'\rightarrow V$ of
the pair $(V,B)$ in the case $(\mu,\nu)=(7,2)$.}  
\label{fig:logkod2-tadam}
\end{figure}

\begin{prop}
For every pair $(\mu,\nu)$ such that $4\nu-\mu=\pm1$, the surface
$S(\mu,\nu)=V\setminus B$ is a $\mathbb{Z}$-acyclic fake plane of
general type, whose real locus is birationally diffeomorphic to $\mathbb{R}^{2}$. \end{prop}
\begin{proof}
By construction $\tau^{*}D=D+\mu A(p)+R$ and $\tau^{*}T_{p_{0}}(D)=T_{p_{0}}(D)+\nu A(p)+R'$
where $R$ and $R'$ are effective divisors supported on the union
of $E_{1},\ldots,E_{r-1}$. The free abelian group $\mathbb{Z}\langle B\rangle$
is generated by the classes of $T_{p_{0}}(D)$, $E_{i},\ldots,E_{r-1}$
and $D-4T_{p_{0}}(D)$ while $\mathrm{Cl}(V)$ is freely generated
by the classes of $T_{p_{0}}(D)$, $E_{1},\ldots,E_{r-1}$ and $A(p)$.
Since $\tau^{*}D\sim4\tau^{*}T_{p_{0}}(D)$ in $\mathrm{Cl}(V)$,
the matrix of the homomorphism $j_{*} \colon\mathbb{Z}\langle B\rangle\rightarrow\mathrm{Cl}(V)$
with respect to these bases has the form 
\[
M=\left(\begin{array}{cc}
\mathrm{id}_{r+1} & *\\
0 & 4\nu-\mu
\end{array}\right)\in\mathcal{M}_{r+2}(\mathbb{Z}).
\]
The hypothesis that $4\nu-\mu=\pm1$ implies that $M\in\mathrm{GL}_{r+2}(\mathbb{Z})$
so that $S$ is $\mathbb{Z}$-acyclic by Theorem \ref{thm:Top-carac}
a). Furthermore, $H^{2}(G,\mathbb{Z}\langle B\rangle)\simeq\mathbb{Z}\langle B\rangle\otimes_{\mathbb{Z}}\mathbb{Z}_{2}$,
$H^{2}(G,\mathrm{Cl}(V))\simeq\mathrm{Cl}(V)\otimes_{\mathbb{Z}}\mathbb{Z}_{2}$
as $\tau$ consists of the blow-ups of real points only, and the homomorphism
$H^{2}(j_{*})$ is represented by the reduction modulo $2$ of $M$.
Since $\mu$ is necessarily odd, $H^{2}(j_{*})$ is thus an isomorphism
and since $B(\mathbb{R})$ is not empty, we conclude from Theorem~\ref{thm:Top-carac} b) that $S$ is a real plane. The fact that $\kappa(S)=2$
follows from the same argument as in the proof of Proposition~\ref{prop:three-generaltype-planes}
above using the structure of the weighted dual graph of the total
transform of $B$ in a real minimal log-resolution $\beta\colon V'\rightarrow V$
of the pair $(V,B)$ (see Figure \ref{fig:logkod2-tadam}).

It remains to show that $S$ is $\mathbb{R}$-biregularly birationally
equivalent to $\mathbb{C}^{2}$. Let $\theta:\mathbb{CP}^{2}\dashrightarrow\mathbb{CP}^{2}$
be the real birational involution consisting of the blow-up of the
points $p_{0}$ with exceptional divisor $L\simeq\mathbb{P}_{\mathbb{R}}^{1}$,
$q$ and $\overline{q}$ followed by the contraction of the proper
transforms of the real line $L_{1}\simeq\mathbb{P}_{\mathbb{R}}^{1}$
passing through $q$ and $\overline{q}$ and of the pair of non-real
conjugate lines passing through $p_{0}$ and $q$ and $p_{0}$ and
$\overline{q}$ respectively. The images of $D$, $T_{p_{0}}(D)$
and $L$ by $\theta$ are respectively a smooth conic $\tilde{D}$,
a real line $\tilde{T}$ intersecting $\tilde{D}$ at $p$ and another
real point $p_{\infty}$, and the tangent line $\tilde{L}$ of $\tilde{D}$
at $p_{\infty}$, see Figure~\ref{fig:tricusp-quartic-rectif} below. 

\begin{figure}[!htb]
\input{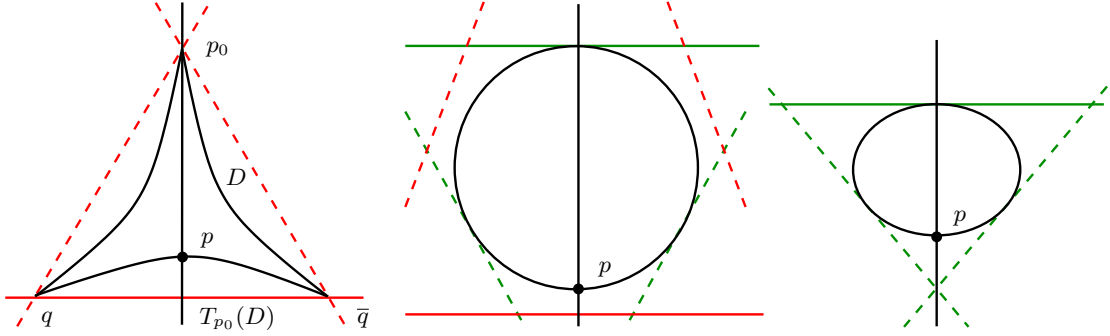} 
\caption{Behavior of the tricuspidal quartic under the standard Cremona transformation.}  
\label{fig:tricusp-quartic-rectif}
\end{figure}

This map $\theta$ lifts to a real birational map
$\tilde{\theta}\colon V\dashrightarrow\tilde{V}$ to the real projective
surface $\tilde{\tau}:\tilde{V}\rightarrow\mathbb{CP}^{2}$ obtained
from $\mathbb{CP}^{2}$ by blowing-up the real point $p_{1}\in\tilde{L}\setminus(\tilde{D}\cap\tilde{L})$
on which $L_{1}$ was contracted and performing the same sequence
of blow-ups as for the construction of $V$ over the point $p\in\tilde{D}\cap\tilde{T}$.
Let $\tilde{L}_{1}$ and $\tilde{E}_{1},\ldots,\tilde{E}{}_{r-1},\tilde{E}{}_{r}=\tilde{A}(p)$
be the corresponding exceptional divisors of $\tilde{\tau}$ over
$p_{1}$ and $p$ respectively and let $\tilde{B}=\tilde{D}\cup\tilde{T}\cup\tilde{L}\cup\bigcup_{i=1}^{r-1}\tilde{E}_{i}$.
Then $\tilde{\theta}$ restricts to a diffeomorphism between the real
loci of the surfaces $S$ and $\tilde{S}=\tilde{V}\setminus\tilde{B}$.
Using the relations 
\[
\begin{cases}
\tilde{\tau}^{*}\tilde{T}=\tilde{\tau}_{*}^{-1}\tilde{T}+\tilde{L}_{1}+\nu\tilde{A}(p)+R_{1} & \sim\tilde{\tau}^{*}\tilde{L}\\
\tilde{\tau}^{*}\tilde{D}=\tilde{\tau}_{*}^{-1}\tilde{D}+\mu\tilde{A}(p)+R_{2} & \sim2\tilde{\tau}^{*}\tilde{L}
\end{cases}
\]
in the divisor class group of $\tilde{V}$, where $R_{1}$ and $R_{2}$
are effective divisors supported on $\bigcup_{i=1}^{r-1}\tilde{E}_{i}$,
we conclude that the homomorphism $\tilde{j}{}_{*}:\mathbb{Z}\langle\tilde{B}\rangle\rightarrow\mathrm{Cl}(\tilde{V})$
induced by the inclusion $\tilde{B}\hookrightarrow\tilde{V}$ is represented
in appropriate bases by a matrix of the form 
\[
\tilde{M}=\left(\begin{array}{ccc}
\mathrm{id}_{r+1} & * & *\\
0 & 0 & -1\\
0 & -\mu & -\nu
\end{array}\right).
\]
Since $4\nu-\mu=\pm1$, $\mu$ is odd, $\tilde{M}\in\mathrm{GL}_{r+3}(\mathbb{Q})$,
and we deduce from Theorem \ref{thm:Top-carac} a) and Remark \ref{Rk:Top-carac-rem}
that $\tilde{S}$ is $\mathbb{Q}$-acyclic, with $H_{1}(\tilde{S},\mathbb{Z})\simeq\mathbb{Z}_{\mu}$.
So $\tilde{S}$ is in particular affine. Furthermore, the proper transform
of $\tilde{L}$ in a minimal real log-resolution $\tilde{\beta}:\tilde{V}'\rightarrow\tilde{V}$
of the pair $(\tilde{V},\tilde{B})$ is a $(-1)$-curve. Contracting
it, we obtain a completion of $\tilde{S}$ into a smooth real projective
surface $(\tilde{V}'',\tilde{B}'')$ whose boundary $\tilde{B}''$
is a chain of rational curves. It follows from \cite{Giz71} that
$\tilde{S}$ admits an $\mathbb{A}^{1}$-fibration over $\mathbb{C}$.
So $\kappa(\tilde{S})=-\infty$, and we then deduce from Theorem \ref{thm:RealPlan-negKappa-Carac}
that $\tilde{S}$ admits a real $\mathbb{A}^{1}$-fibration $\tilde{\rho}:\tilde{S}\rightarrow\mathbb{C}$.
By \cite[Proposition 2.15]{Du05}, $\tilde{\rho}$ has at most one
degenerate fiber, and so, $\tilde{S}$ is $\mathbb{R}$-biregularly
equivalent to $\mathbb{C}^{2}$ by virtue of Theorem \ref{thm:Real-plane-NegKappa-Rectif}.
This completes the proof.
\end{proof}

\begin{example}
In the simplest case $(\mu,\nu)=(3,1)$, the dual graph of the total
transform of $\tilde{L}_{1}\cup\tilde{B}\cup\tilde{A}(p)$ in the
minimal log-resolution $\tilde{\beta}\colon \tilde{V}'\rightarrow\tilde{V}$
of the pair $(\tilde{V},\tilde{B}=\tilde{D}\cup\tilde{T}\cup\tilde{L}\cup\tilde{E}_{1}\cup\tilde{E}_{2})$
is depicted on the right hand side of Figure~\ref{fig:logkod2-tadam-rectif},
\begin{figure}[!htb]
\input{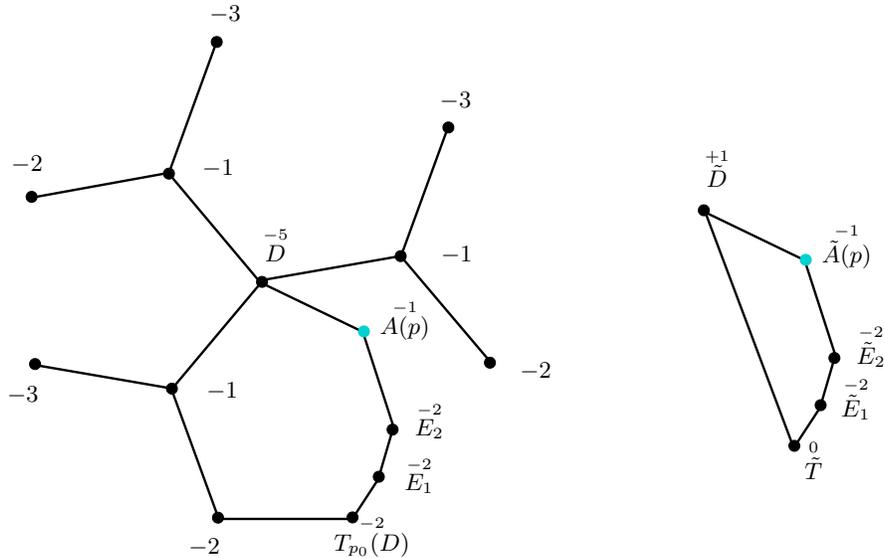} 
\caption{Dual graphs of the total transforms of $B\cup A(p)$ and $\tilde B\cup \tilde A(p)\cup\tilde{L}_1$ in a minimal log-resolution of the pairs $(V, B)$ and $(\tilde{V},\tilde{B})$ respectively for $(\mu,\nu)=(3,1)$.}  
\label{fig:logkod2-tadam-rectif}
\end{figure}
where $E_{p_{\infty},1}$ and $E_{p_{\infty},2}$ denote the two exceptional
divisors of $\tilde{\beta}$ over $p_{\infty}$. Let $\gamma:\tilde{V}'\dashrightarrow\tilde{V}''$
be the real birational map consisting of the blow-up of the real point
$E_{p_{\infty,1}}\cap E_{p_{\infty},2}$ with exceptional divisor
$C$, followed by the contraction of $\tilde{L}$ and $\tilde{D}$.
The total image $\tilde{B}''=E_{p_{\infty},2}\cup C\cup E_{p_{\infty},1}\cup\tilde{T}\cup\tilde{E}_{1}\cup\tilde{E}_{2}$
of $\tilde{B}'$ by $\gamma$ is a chain of rational curves with self-intersections
$(0,-1,-3,-2,-2,-2)$. The complete linear system $|E_{p_{\infty,2}}|$
on $\tilde{V}''$ defines a real $\mathbb{P}^{1}$-fibration $\overline{\rho}\colon \tilde{V}''\rightarrow\mathbb{CP}^{1}$
having $C$ has a section and $E_{p_{\infty,1}}+3\tilde{T}+3\tilde{L}_{1}+2\tilde{E}_{1}+\tilde{E}_{2}$
as a unique degenerate fiber. The restriction of $\overline{\rho}$
to $\tilde{S}\simeq\tilde{V}''\setminus\tilde{B}''$ is a real $\mathbb{A}^{1}$-fibration
$\tilde{\rho}:\tilde{S}\rightarrow\mathbb{C}$ having $\tilde{L}_{1}\cap\tilde{S}\simeq\mathbb{C}$
has unique degenerate fiber, of multiplicity $3$.
\end{example}

\bibliographystyle{amsplain} 

\providecommand{\bysame}{\leavevmode\hbox to3em{\hrulefill}\thinspace}
\providecommand{\MR}{\relax\ifhmode\unskip\space\fi MR }
\providecommand{\MRhref}[2]{%
  \href{http://www.ams.org/mathscinet-getitem?mr=#1}{#2}
}
\providecommand{\href}[2]{#2}
\begin{thebibliography}{}

\end{thebibliography}


\begin{thebibliography}{99}
\bibitem{Be83} J. Bertin, \emph{Pinceaux de droites et automorphismes des surfaces affines}, J. Reine Angew. Math. 341 (1983), 32-53.

\bibitem{BD11} J. Blanc and A. Dubouloz, \emph{Automorphisms of $\mathbb{A}^1$-fibered surfaces}, Trans. Amer. Math. Soc. 363 (2011), 5887-5924. 


\bibitem{BH61} A.~Borel and  A.~Haefliger, \emph{La classe d'homologie fondamentale d'un espace analytique}, Bull. Soc. Math. France, 89 (1961), 461-513.

\bibitem{tDPe90} T. tom Dieck and T. Petrie, \emph{Contractible affine surfaces of Kodaira dimension one}, Japan J. Math. 16 (1990), 147-169.

\bibitem{tDPe93} T. tom Dieck and T. Petrie, \emph{Homology planes and algebraic curves}, Osaka J. Math. 30 (1993), no. 4, 855-886.

\bibitem{Dimca} A. Dimca, \emph{Singularities and Topology of Hypersurfaces}, Universitext, Springer, 1992.

\bibitem{Du05}  A. Dubouloz, \emph{Completions of normal affine surfaces with a trivial Makar-Limanov invariant}, Michigan Math. J.  52 (2004), no. 2, 289-308. 

\bibitem{DM15} A.~Dubouloz, F.~Mangolte, \emph{Fake real planes: exotic affine algebraic models of $\mathbb{R}^{2}$}, \emph{Submitted} (2015), \texttt{arXiv:1507.01574 [math.AG]}, 36 pages, 18 figures.

\bibitem{Fu82} T. Fujita, \emph{On the topology of noncomplete algebraic surfaces}, J. Fac. Sci. Univ. Tokyo Sect. IA Math. 29 (1982), no. 3, 503-566.

\bibitem{Giz71} M. H. Gizatullin, Quasihomogeneous affine surfaces, Izv. Akad. Nauk SSSR Ser. Mat. 35 (1971), 1047-1071.

\bibitem{GuMi87} R.V. Gurjar and M. Miyanishi, \emph{Affine surfaces with $\overline{\kappa}\leq 1$}, Algebraic Geometry and Commutative Algebra, in honor of M. Nagata, 1987, 99-124.

\bibitem{GuP97} V. Gurjar and C. R. Pradeep, \emph{$\mathbb{Q}$-homology planes are rational. III}, Osaka J. Math. 36 (1999), no. 2, 259-335. 

\bibitem{GuPS97} R. V. Gurjar, C. R. Pradeep, and A. R. Shastri, \emph{On rationality of logarithmic $\mathbb{Q}$-homology planes. II}, Osaka J. Math. 34 (1997), no. 3, 725-743.

\bibitem{Kam75} T. Kambayashi, \emph{On the absence of nontrivial separable forms of the affine plane}, J. of Algebra 35 (1975), 449-456.

\bibitem{Ko01real} J. Koll\'ar, \emph{The topology of real algebraic varieties}, in \emph{Current   developments in mathematics, 2000}, Int. Press, Somerville, MA, 2001,  197-231.

\bibitem{KK02}
V.~S. Kulikov and  V.~M. Kharlamov,  \emph{On
  real structures on rigid surfaces}, Izv. Ross. Akad. Nauk Ser. Mat. 66 (2002), 133-152.

\bibitem{Ma15} F. Mangolte, \emph{Real rational surfaces}, in \emph{Real Algebraic Geometry}, Panoramas et synthèses, 2015,  \texttt{arXiv:1503.01248 [math.AG]},  to appear, 29 pages, 6 figure. 

\bibitem{MiyBook} M. Miyanishi, \emph{Open {A}lgebraic {S}urfaces}, CRM Monogr. Ser., 12, Amer. Math. Soc., Providence, RI, 2001. 

\bibitem{MS80} M. Miyanishi and T. Sugie, \emph{Affine surfaces containing cylinderlike open sets}, J. Math. Kyoto Univ. 20, no. 1 (1980), 11-42. 

\bibitem{Na83} M. Namba, \emph{Geometry of projective algebraic curves}, Monograph and Textbooks in Pure and Applied Mathematics, vol. 88, Marcel Dekker Inc., New York, 1984.

\bibitem{Ram71} C. P. Ramanujam, \emph{A  topological  characterisation  of  the  affine  plane  as  an  algebraic variety}, Ann. of Math. 94 (1971), 69-88.


\bibitem{Sil89}  R.~Silhol, \emph{Real algebraic surfaces}, Lecture Notes in Mathematics, vol.   1392, Springer-Verlag, Berlin, 1989.

\bibitem{Wak78} I. Wakabayashi,  \emph{On the logarithmic Kodaira dimension of the complement of a curve in $\mathbb{P}^2$}, Proc. Japan Acad. 54 A (1978), 157-162.

\bibitem{Zai00} M. Zaidenberg, \emph{Exotic algebraic structures on affine spaces}, St. Petersburg Math. J. 11 (5) (2000), 703-760. 


\end{thebibliography}

\end{document}